\documentclass[a4paper, 11pt, twoside]{amsart}

\usepackage[all]{xy}
\usepackage{amsmath}
\usepackage{amssymb}
\usepackage{latexsym}
\usepackage{enumerate}
\usepackage{prettyref}
\usepackage{hyperref}
\usepackage{color}
\usepackage[lite, alphabetic]{amsrefs}
\usepackage[margin=1in]{geometry}

\newtheorem{mainthm}{Theorem}
\newtheorem{theorem}{Theorem}[section]
\newrefformat{thm}{\hyperref[{#1}]{Theorem~\ref*{#1}}}

\newrefformat{lem}{\hyperref[{#1}]{Lemma~\ref*{#1}}}
\newtheorem{proposition}[theorem]{Proposition}
\newrefformat{prop}{\hyperref[{#1}]{Proposition~\ref*{#1}}}
\newtheorem{corollary}[theorem]{Corollary}
\newrefformat{cor}{\hyperref[{#1}]{Corollary~\ref*{#1}}}
\newtheorem{conjecture}[theorem]{Conjecture}

\theoremstyle{definition}
\newtheorem{definition}[theorem]{Definition}
\newtheorem{remark}[theorem]{Remark}
{
\newtheorem{examplecore}[theorem]{Example}}

\newenvironment{example}{\begin{examplecore}}{\hspace*{\fill}
$\square$\par\vspace{.1cm}\end{examplecore}}

\newcommand{\op}{\operatorname}

\begin{document}

\title{On Farrell--Tate cohomology of $\op{SL}_2$ over $S$-integers}  

\author{Alexander D. Rahm and Matthias Wendt}

\address{Alexander D. Rahm, 
 Mathematics Research Unit of Universit\'e du Luxembourg,
6, rue Richard Coudenhove-Kalergi,
L-1359 Luxembourg }
\email{Alexander.Rahm@uni.lu}
\address{Matthias Wendt, Fakult\"at Mathematik,
Universit\"at Duisburg-Essen, Thea-Leymann-Strasse 9, 45127, Essen,
Germany}
\email{matthias.wendt@uni-due.de}


\date{\today}

\begin{abstract}
In this paper, we provide number-theoretic formulas for Farrell--Tate cohomology for $\op{SL}_2$ over rings of $S$-integers in number fields satisfying a weak regularity assumption. These formulas describe group cohomology above the virtual cohomological dimension, and can be used to study some questions in homology of linear groups.

We expose three applications, to (I) detection questions for the Quillen conjecture, 
\\(II) the existence of transfers for the Friedlander--Milnor conjecture, 
\\(III) cohomology of $\op{SL}_2$ over number fields. 
\end{abstract}

\maketitle

\setcounter{tocdepth}{1}

\section{Introduction}
\label{sec:intro}

The cohomology of arithmetic groups like $\op{SL}_2(\mathcal{O}_{K,S})$ for
$\mathcal{O}_{K,S}$ a ring of $S$-integers in a number field $K$ has been long and intensively studied. In principle, the cohomology groups can be computed from the action of $\op{SL}_2$ on its associated symmetric space, but actually carrying out this program involves a lot of questions and difficulties from algebraic number theory. 
In the case of a group $\Gamma$ of finite virtual cohomological dimension (vcd), Farrell--Tate cohomology provides a modification of group cohomology, which in a sense describes the obstruction for $\Gamma$ to be a Poincar{\'e}-duality group. The Farrell--Tate cohomology of $\Gamma$ can be described in terms of finite subgroups of $\Gamma$ and their normalizers, and hence is more amenable to computation than group cohomology.

The primary goal of the present paper is to provide explicit formulas for the Farrell--Tate cohomology of $\op{SL}_2$ over many rings of $S$-integers in number fields, with coefficients in $\mathbb{F}_\ell$, $\ell$  odd. These formulas generalize those which had previously been obtained in the case of imaginary quadratic number rings by one of the authors, building on work of Kr\"amer, cf. \cites{rahm:torsion,rahm:vcd}, \cite{kraemer:diplom}.

We will state the main result after introducing some relevant notation. Let $K$ be a global field, let $S$ be a non-empty set of places containing the infinite ones and denote by $\mathcal{O}_{K,S}$ the ring of $S$-integers in $K$. Let $\ell$ be an odd prime different from the  characteristic of $K$, and let $\zeta_\ell$ be some primitive $\ell$-th root of unity. 

Assume $\zeta_\ell+\zeta_\ell^{-1}\in K$. We denote by $\Psi_\ell(T)=T^2-(\zeta_\ell+\zeta_\ell^{-1})T+1$ the relevant quadratic factor of the $\ell$-th cyclotomic polynomial, by $K(\Psi_\ell)=K[T]/\Psi_\ell(T)$ the corresponding $K$-algebra, by $R_{K,S,\ell}=\mathcal{O}_{K,S}[T]/(\Psi_\ell(T))$  the corresponding order in $K(\Psi_\ell)$, and by 
$$
\op{Nm}_1:R_{K,S,\ell}^\times\to
  \mathcal{O}_{K,S}^\times\qquad\textrm{ and }\qquad
\op{Nm}_0:
  \op{Pic}(R_{K,S,\ell})\to 
  \op{Pic}(\mathcal{O}_{K,S}).
$$
the norm maps on unit groups and class groups for the finite extension $R_{K,S,\ell}/\mathcal{O}_{K,S}$.

With this notation, the following is our main result; the main part of the proof can be found in Section~\ref{sec:formulas}. 

\begin{mainthm}
\label{thm:main1}
Let $\ell$ be an odd prime. Let $K$ be a global field of characteristic different from $\ell$ which contains $\zeta_\ell+\zeta_\ell^{-1}$, let $S$ be a non-empty set of places containing the infinite ones and denote by $\mathcal{O}_{K,S}$ the ring of $S$-integers in $K$. Furthermore, assume one of the following conditions: 
\begin{enumerate}[({R}1)]
\item $\zeta_\ell\not\in K$ and the prime $(\zeta_\ell-\zeta_\ell^{-1})$ is unramified in the extension $\mathcal{O}_{K,S}/\mathbb{Z}[\zeta_\ell+\zeta_\ell^{-1}]$. 
\item $\zeta_\ell\in K$ and $S_\ell\in S$, i.e., $S$ contains the places of $K$ lying over $\ell$.
\end{enumerate}
Then we have the following consequences:
\begin{enumerate}
\item $\widehat{\op{H}}^\bullet(\op{SL}_2(\mathcal{O}_{K,S}),\mathbb{F}_\ell)\neq
  0$.
\item The set $\mathcal{C}_\ell$ of conjugacy classes of elements in $\op{SL}_2(\mathcal{O}_{K,S})$ with characteristic polynomial $\Psi_\ell(T)=T^2-(\zeta_\ell+\zeta_\ell^{-1})T+1$ sits in an extension 
$$
1\to \op{coker}\op{Nm}_1\to \mathcal{C}_\ell\to 
\ker\op{Nm}_0\to 0.
$$
The set $\mathcal{K}_\ell$ of conjugacy classes of order $\ell$ subgroups of $\op{SL}_2(\mathcal{O}_{K,S})$ can be identified with the orbit set $\mathcal{K}_\ell=\mathcal{C}_\ell/\op{Gal}(K(\Psi_\ell)/K)$, where the action of $\op{Gal}(K(\Psi_\ell)/K)\cong\mathbb{Z}/2\mathbb{Z}$ exchanges the two roots of $\Psi_\ell(T)$. There is a product decomposition  
$$
\widehat{\op{H}}^\bullet(\op{SL}_2(\mathcal{O}_{K,S}),\mathbb{F}_\ell)\cong  
\prod_{[\Gamma]\in\mathcal{K}_\ell}
\widehat{\op{H}}^\bullet(N_{\op{SL}_2(\mathcal{O}_{K,S})}(\Gamma),
\mathbb{F}_\ell). 
$$
This decomposition is compatible with the ring structure of $\widehat{\op{H}}^\bullet(\op{SL}_2(\mathcal{O}_{K,S}),\mathbb{F}_\ell)$.
\item If the class of  $\Gamma$ is not
  $\op{Gal}(K(\Psi_\ell)/K)$-invariant, 
then 
$
N_{\op{SL}_2(\mathcal{O}_{K,S})}(\Gamma)\cong \ker\op{Nm}_1.
$
There is an isomorphism of graded rings
$$
\widehat{\op{H}}^\bullet(N_{\op{SL}_2(\mathcal{O}_{K,S})}(\Gamma),
\mathbb{Z})_{(\ell)}\cong
\mathbb{F}_\ell[a_2,a_2^{-1}]\otimes_{\mathbb{F}_\ell}\bigwedge 
\left(\ker\op{Nm}_1\right),
$$
where $a_2$ is a cohomology class of degree $2$.
In particular, this is a free module over the subring $\mathbb{F}_\ell[a_2^2,a_2^{-2}]$.
\item
 If the class of  $\Gamma$ is 
  $\op{Gal}(K(\Psi_\ell)/K)$-invariant, 
then there is an extension
$$
0\to \ker\op{Nm}_1\to N_{\op{SL}_2(\mathcal{O}_{K,S})}(\Gamma)\to
\mathbb{Z}/2\mathbb{Z}\to 1. 
$$
There is an isomorphism of graded rings
$$
\widehat{\op{H}}^\bullet(N_{\op{SL}_2(\mathcal{O}_{K,S})}(\Gamma),
  \mathbb{Z})_{(\ell)}\cong
\left(\mathbb{F}_\ell[a_2,a_2^{-1}]\otimes_{\mathbb{F}_\ell}\bigwedge
\left(\ker\op{Nm}_1\right)\right)^{\mathbb{Z}/2}, 
$$
with the $\mathbb{Z}/2$-action given by multiplication with $-1$ on
$a_2$ and $\ker\op{Nm}_1$. In particular, this is a free module over
the subring
$\mathbb{F}_\ell[a_2^2,a_2^{-2}]\cong
\widehat{\op{H}}^\bullet(D_{2\ell},\mathbb{Z})_{(\ell)}$. 
\end{enumerate}
\end{mainthm}

\begin{remark}
The conditions (R1) and (R2) exclude pathologies with the singularities over  the place $\ell$ and resulting non-invertibility of ideals. For global function fields, they are automatically satisfied whenever $\ell$ is different from the characteristic. Condition (R2) is a standard assumption when dealing with the Quillen conjecture, which is one of the relevant applications of Theorem~\ref{thm:main1}.  
\end{remark}

\begin{remark}
We recover as special cases the earlier results  of Busch
\cite{busch:conjugacy}. For function fields of curves over
algebraically closed fields, there are similar formulas,
cf. \cite{sl2parabolic}.
\end{remark}

\begin{remark}
The restriction to odd torsion coefficients is necessary. Only under
this assumption, the formula has such a simple structure; for
$2$-torsion one has to take care of the possible subgroups $A_4$,
$S_4$ and $A_5$ in $\op{PSL}_2$. Moreover, for
$\mathbb{F}_2$-coefficients, there is a huge difference between the
cohomology of $\op{SL}_2$ and $\op{PSL}_2$ - only the 
cohomology of $\op{PSL}_2$ has an easy description and the computation
of cohomology of $\op{SL}_2$ depends on a lot more than just the
finite subgroups. Some results can be achieved, but the additional
complications would only obscure the presentation of the results.  
\end{remark}

The main strategy of proof is the natural one: we use Brown's formula which computes the Farrell--Tate cohomology with $\mathbb{F}_\ell$-coefficients of $\op{SL}_2(\mathcal{O}_{K,S})$ in terms of the elementary abelian $\ell$-subgroups and their normalizers. The main piece of information for the above theorem is then the classification of conjugacy classes of finite subgroups of $\op{SL}_2(\mathcal{O}_{K,S})$, which -- under the conditions (R1) or (R2) -- is directly related to number-theoretic questions about (relative) class groups and unit groups of $S$-integers in $K$ and its cyclotomic extension. Another secondary objective of the current paper is to reinterpret and generalize results on finite subgroup classification in number-theoretic terms.

The explicit formulas obtained allow to discuss a couple of questions concerning cohomology of linear groups. These applications are based on the fact that Farrell--Tate cohomology and group cohomology coincide above the virtual cohomological dimension. In particular, Theorem~\ref{thm:main1} provides a computation of group cohomology of $\op{SL}_2(\mathcal{O}_{K,S})$ above the virtual cohomological dimension. 

As a first application, the explicit formulas for the ring structure allow to discuss a conjecture of Quillen, cf. \cite{quillen:spectrum}, recalled as Conjecture~\ref{Quillen-conjecture} in Section~\ref{sec:quillen} below. 
The following result is proved in Section~\ref{sec:quillen}; 
conclusions from it with consequences for Quillen's conjecture have been drawn in~\cite{RW14}.

\begin{mainthm} 
\label{thm:main3}
With the notation of Theorem~\ref{thm:main1}, the restriction map induced from the inclusion $\op{SL}_2(\mathcal{O}_{K,S})\to \op{SL}_2(\mathbb{C})$ maps the second Chern class $c_2\in\op{H}^\bullet_{\op{cts}}(\op{SL}_2(\mathbb{C}),\mathbb{F}_\ell)$ to the sum of the elements $a_2^2$ in all the components. 

As a consequence, Quillen's conjecture for $\op{SL}_2(\mathcal{O}_{K,S})$ is true for Farrell--Tate cohomology with ${\mathbb{F}_\ell}$-coefficients for all $K$ and $S$. This also implies that Quillen's conjecture is true for group cohomology with ${\mathbb{F}_\ell}$-coefficients above the virtual cohomological dimension.  

However, if the number of conjugacy classes of order $\ell$-subgroups is greater than two, the restriction map 
$$
\op{H}^\bullet(\op{SL}_2(\mathcal{O}_{K,S}),\mathbb{F}_\ell)\to
\op{H}^\bullet(\op{T}_2(\mathcal{O}_{K,S}),\mathbb{F}_\ell)
$$
from $\op{SL}_2(\mathcal{O}_{K,S})$ to the group $\op{T}_2(\mathcal{O}_{K,S})$ of diagonal matrices is not injective. 
\end{mainthm}

This result sheds light on the Quillen conjecture, showing that the Quillen conjecture holds in a number of rank one cases; it also sheds light on the relation between Quillen's conjecture and detection questions in group cohomology, where detection refers to injectivity of the restriction map to diagonal matrices. See \cite{RW14} for a more detailed discussion of these issues. 
It is also worthwhile pointing out that the failure of detection as in the previous theorem, i.e., the non-injectivity of the restriction map 
$$
\op{H}^\bullet(\op{SL}_2(\mathcal{O}_{K,S}),\mathbb{F}_\ell)\to
\op{H}^\bullet(\op{T}_2(\mathcal{O}_{K,S}),\mathbb{F}_\ell)
$$
implies the failure of the unstable Quillen-Lichtenbaum conjecture as formulated in \cite{anton:roberts}. This follows from the work of Dwyer and Friedlander \cite{dwyer:friedlander}. Other examples of the failure of detection in function fields situations in a higher rank case are discussed in \cite{ellnote}.

Another interesting question in group cohomology is the existence of transfers, which was discussed in \cite{knudson:book}*{Section~5.3} in the context of the Friedlander--Milnor conjecture. From the explicit computations of Theorem~\ref{thm:main1}, we can also describe the restriction maps on Farrell--Tate cohomology which  allows to find  examples for non-existence of  transfers. The following result is proved in Section~\ref{sec:transfers}:

\begin{mainthm}
\label{thm:main4}
Let $L/K$ be a finite separable extension of global fields, let $S$ be a non-empty finite set of places of $K$ containing the infinite places and let $\tilde{S}$ be a set of places of $L$ containing those lying over $S$. Let $\ell$ be an odd prime different from the characteristic of $K$.  Assume that $\mathcal{O}_{K,S}$ and $\mathcal{O}_{L,\tilde{S}}$ satisfy the conditions of Theorem~\ref{thm:main1}. 

The restriction map  
$$
\widehat{\op{H}}^\bullet(\op{SL}_2(\mathcal{O}_{L,\tilde{S}}),\mathbb{F}_\ell)\to 
\widehat{\op{H}}^\bullet(\op{SL}_2(\mathcal{O}_{K,S}),\mathbb{F}_\ell)
$$
induced from the natural ring homomorphism $\mathcal{O}_{K,S}\to\mathcal{O}_{L,\tilde{S}}$ is compatible with the decomposition of Theorem~\ref{thm:main1} and is completely described by  \begin{enumerate}
\item the induced map on class groups $\op{Pic}(R_{K,S,\ell})\to \op{Pic}(R_{L,\tilde{S},\ell})$
\item the induced map on unit groups $R_{K,S,\ell}^\times\to R_{L,\tilde{S},\ell}^\times$.
\end{enumerate}

Let $K=\mathbb{Q}(\zeta_\ell)$ be a cyclotomic field whose class group is non-trivial and has  order prime to $\ell$, e.g. $\ell=23$. Denoting by $H$ the Hilbert class field of $K$, the restriction map  
$$
\widehat{\op{H}}^\bullet(\op{SL}_2(\mathcal{O}_H[\ell^{-1}]),\mathbb{F}_\ell)\to 
\widehat{\op{H}}^\bullet(\op{SL}_2(\mathcal{O}_K[\ell^{-1}]),\mathbb{F}_\ell) 
$$
is not surjective. Therefore, this is an example of a finite \'etale morphism for which it is not possible to define a transfer in the usual K-theoretic sense. 
\end{mainthm}

Finally, we want to note that the precise description of Farrell--Tate cohomology and the relevant restriction maps allows to consider the colimit of the groups $\widehat{\op{H}}^\bullet(\op{SL}_2(\mathcal{O}_{K,S}),\mathbb{F}_\ell)$, where $S$ runs through the finite sets of places of $K$. Using this, we investigate in Section~\ref{sec:borel} the behaviour of Mislin's extension of Farrell--Tate cohomology with respect to directed colimits: 

\begin{mainthm} ${}$
\label{thm:main5}
\begin{enumerate} 
\item There are cases where Mislin's extension of Farrell--Tate cohomology does not commute with directed colimits. A simple example is given by the directed system of groups $\op{SL}_2(\mathbb{Z}[1/n])$ with $n\in \mathbb{N}$. 
\item The Friedlander--Milnor conjecture for $\op{SL}_2(\overline{\mathbb{Q}})$ is equivalent to the question if Mislin's extension of Farrell--Tate cohomology commutes with the directed colimit of $\op{SL}_2(\mathcal{O}_{K,S})$, where $K$ runs through all number fields and $S$ runs through all finite sets of places.
\end{enumerate}
\end{mainthm}

Results similar to the ones in the present paper can be obtained for higher rank groups, although there are more and more complications coming from the classification of finite subgroups. Away from the order of the Weyl group, the subgroup classification is easier. A discussion of the case $\op{SL}_3$ will be done in the forthcoming paper~\cite{sl3ok}. 

\emph{Structure of the paper:} We first recall group cohomology preliminaries in Section~\ref{sec:pregrp}. The proof of the main theorem, modulo the conjugacy classification, is given in Section~\ref{sec:formulas}. 
Sections~\ref{sec:conjugacy1} and \ref{sec:conjugacy2} establish the conjugacy classification of finite cyclic subgroups in $\op{SL}_2(\mathcal{O}_{K,S})$. Then we discuss three applications of the results, to (I) non-detection in Section~\ref{sec:quillen}, (II) the existence of transfers in Section~\ref{sec:transfers} and (III) cohomology of SL$_2$ over number fields in Section~\ref{sec:borel}.  

\emph{Acknowledgements:} This work was started in August 2012 during a
stay of the second named author at the De Br\'un Center for
Computational Algebra at NUI Galway. We would like thank Guido Mislin
for email correspondence concerning his version of Tate cohomology,
and Jo\"el Bella\"iche for an enlightening MathOverflow answer
concerning group actions on Bruhat--Tits trees. We thank Norbert Kr\"amer and an anonymous referee for very helpful comments on a previous version of this paper.

\section{Preliminaries on group cohomology and Farrell-Tate
  cohomology} 
\label{sec:pregrp}

In this section, we recall the necessary definitions of group cohomology and Farrell--Tate cohomology, introducing notations and results we will need in the sequel. One of the basic references is \cite{brown:book}. 

\subsection{Group cohomology}

Group cohomology is defined as the right derived functor of invariants $\mathbb{Z}[G]\op{-mod}\to \mathbb{Z}\op{-mod}:M\mapsto M^G$. It can be defined algebraically by taking a resolution $P_\bullet\to \mathbb{Z}$ of $\mathbb{Z}$ by projective $\mathbb{Z}[G]$-modules and setting 
$$
\op{H}^\bullet(G,M):=\op{H}^\bullet(\op{Hom}_G(P_\bullet,M)).
$$
Alternatively, it can be defined topologically as the cohomology of the classifying space $BG$ with coefficients in the local system associated to $M$.

\subsection{Farrell--Tate cohomology}

We shortly recall the definition and  properties of Farrell--Tate cohomology, cf. \cite{brown:book}*{chapter X}. Farrell--Tate cohomology is a completion of group cohomology defined for groups of finite virtual cohomological dimension (vcd). Note that for $K$ a number field, the groups $\op{SL}_2(\mathcal{O}_{K,S})$  have virtual cohomological dimension $2r_1+3r_2+\# S_f-1$, where $r_1$ and 2$r_2$ are the numbers of real and complex embeddings of $K$, respectively, and $\# S_f$ is the number of finite places in $S$.

For $\Gamma$ a group with finite virtual cohomological dimension, a
complete resolution is the datum of 
\begin{itemize}
\item an acyclic chain complex $F_\bullet$ of projective $\mathbb{Z}[\Gamma]$-modules, 
\item a projective resolution $\epsilon:P_\bullet\to\mathbb{Z}$ of $\mathbb{Z}$ over $\mathbb{Z}[\Gamma]$, and
\item a chain map $\tau:F_\bullet\to P_\bullet$ which is the identity in sufficiently high dimensions.
\end{itemize}

The starting point of Farrell--Tate cohomology is the fact that groups of finite virtual cohomological dimensions have complete resolutions. 

\begin{definition}
Given a group $\Gamma$ of finite virtual cohomological dimension, a complete resolution $(F_\bullet, P_\bullet, \epsilon)$ and a $\mathbb{Z}[\Gamma]$-module $M$, \emph{Farrell--Tate cohomology of $\Gamma$ with coefficients in $M$} is defined by 
$$
\widehat{\op{H}}^\bullet(\Gamma,M):=
\op{H}^\bullet(\op{Hom}_\Gamma(F_\bullet, M)).
$$
\end{definition}

The functors $\widehat{\op{H}}^\bullet$ satisfy all the usual cohomological properties, cf. \cite{brown:book}*{X.3.2}, and in fact Farrell--Tate cohomology can be seen as projective completion of group cohomology, cf. \cite{mislin:tate}. Important for our considerations is the fact \cite{brown:book}*{X.3.4} that there is a canonical map 
$$
\op{H}^\bullet(\Gamma,M)\to \widehat{\op{H}}^\bullet(\Gamma,M)
$$
which is an isomorphism above the virtual cohomological dimension.

\subsection{Finite subgroups in $\op{PSL}_2(K)$}
We first recall the classification of finite subgroups of $\op{PSL}_2(K)$, $K$ any field. \emph{We  implicitly assume that the order of the finite subgroup is prime to the  characteristic of $K$.} The classification over algebraically closed fields is due to Klein.  In the (slightly different) case of $\op{PGL}_2(K)$, $K$ a general field, the classification of finite subgroups can be found in \cite{beauville}.  

For $K$ an algebraically closed field, in particular for $K=\mathbb{C}$, Klein's classification provides an exact list of isomorphism types (as well as conjugacy classes) of finite subgroups in $\op{PSL}_2(K)$: any finite subgroup of $\op{PSL}_2(\mathbb{C})$ is isomorphic to a cyclic group $\mathbb{Z}/n\mathbb{Z}$, a dihedral group $D_{2n}$, the tetrahedral group $A_4$, the octahedral group $S_4$ or the icosahedral group $A_5$. 

Over an arbitrary field $K$, we have the following classification, cf. \cite{serre:finite}*{2.5}, \cite{kraemer:diplom}*{Satz 13.3} or \cite{beauville}*{proposition 1.1 and theorem 4.2}. Denote by $\zeta_n$ some primitive $n$-th root of unity.

\begin{proposition}
\label{prop:psl2finite}
\begin{enumerate}[(i)]
\item $\op{PSL}_2(K)$ contains a cyclic group $\mathbb{Z}/n\mathbb{Z}$ if and only if $\zeta_{2n}+\zeta_{2n}^{-1}\in K$. 
\item 
$\op{PSL}_2(K)$ contains a dihedral group $D_{2n}$ if additionally the symbol $\left((\zeta_{2n}-\zeta_{2n}^{-1})^2,-1\right)$ is split.
\item $\op{PSL}_2(K)$ contains $A_4$ if and only if $-1$ is a sum of two squares,  i.e., if the symbol $(-1,-1)$ is split. 
\item 
$\op{PSL}_2(K)$ contains $S_4$ if and only if  $\sqrt{2}\in K$ and $-1$ is a sum of two squares. 
\item $\op{PSL}_2(K)$ contains $A_5$ if and only if $\sqrt{5}\in K$ and $-1$ is a sum of two squares.
\end{enumerate}
\end{proposition}

This result is best proved by considering the Wedderburn decomposition of $K[G]$ and checking for (determinant one) two-dimensional representations among the factors. 

With the exception of $\mathbb{Z}/2\mathbb{Z}$ and the dihedral groups, all finite subgroups in $\op{PGL}_2(K)$ are conjugate whenever they are isomorphic. For the dihedral groups $D_r$, there is a bijection between $\op{PGL}_2(K)$-conjugacy classes and $K^\times/((K^\times)^2\cdot\mu_r(K))$ if $\zeta_r\in K$, cf. \cite{beauville}*{theorem 4.2}.

\subsection{Farrell--Tate cohomology of finite  subgroups of $\op{SL}_2(\mathbb{C})$} 

We recall the well-known formulas for group and Tate cohomology of cyclic and dihedral groups. We restrict to the cohomology with odd torsion coefficients, as our main results only use that case. The formulas below as well as corresponding formulas for the cohomology with $\mathbb{F}_2$-coefficients can be found in \cite{adem:milgram}.
Here, classes in square brackets are polynomial generators and classes in parentheses are exterior classes; the index of a class specifies its degree in the graded $\mathbb{F}_\ell$-algebra.

\begin{itemize}
\item 
The cohomology ring for a cyclic group of order $n$ with $\ell\mid n$ and $\ell$ odd is given by the formula 
$$
\op{H}^\bullet(\mathbb{Z}/n\mathbb{Z},\mathbb{F}_\ell)\cong
\mathbb{F}_\ell[a_2](b_1). 
$$
The corresponding Tate cohomology ring is $\widehat{\op{H}}^\bullet(\mathbb{Z}/n\mathbb{Z},\mathbb{F}_\ell)\cong \mathbb{F}_\ell[a_2,a_2^{-1}](b_1)$.
Note that cohomology with integral coefficients gets rid of the exterior algebra contribution which come from the universal coefficient formula. 
\item
The cohomology ring for a dihedral group of order $2n$ with $\ell\mid n$ and $\ell$ odd is given by the formula 
$$
\op{H}^\bullet(D_{2n},\mathbb{F}_\ell)\cong
\mathbb{F}_\ell[a_4](b_3). 
$$
The corresponding Tate cohomology ring is $\widehat{\op{H}}^\bullet(D_{2n},\mathbb{F}_\ell)\cong \mathbb{F}_\ell[a_4,a_4^{-1}](b_3)$.
\item The inclusions $D_6\hookrightarrow A_4$, $D_6\hookrightarrow S_4$ and $D_6\hookrightarrow A_5$ induce isomorphisms in group cohomology with $\mathbb{F}_3$-coefficients.
\item The subgroup inclusion $D_{10}\to A_5$ induces an isomorphism in group cohomology with $\mathbb{F}_5$-coefficients. 
\end{itemize}

\section{Computation of Farrell--Tate cohomology}
\label{sec:formulas}

\subsection{Brown's formula}
We now outline the proof of the main result, Theorem~\ref{thm:main1}. The essential tool is Brown's formula. For $\ell$ an odd prime, any elementary abelian $\ell$-subgroup of $\op{SL}_2(\mathcal{O}_{K,S})$ is in fact cyclic. This implies that Brown's complex of elementary abelian $\ell$-subgroups is a disjoint union of the conjugacy classes of cyclic $\ell$-subgroups of $\op{SL}_2(\mathcal{O}_{K,S})$. 

By Brown's formula for Farrell--Tate cohomology, cf. \cite{brown:book}*{corollary X.7.4}, we have 
$$
\widehat{\op{H}}^\bullet(\op{SL}_2(\mathcal{O}_{K,S}),\mathbb{F}_\ell)\cong
\prod_{[\Gamma\leq \op{SL}_2], \medspace \Gamma \textrm{ cyclic}}
\widehat{\op{H}}^\bullet(C_{\op{SL}_2}(\Gamma),
\mathbb{F}_\ell)^{N_{\op{SL}_2}(\Gamma)/C_{\op{SL}_2}(\Gamma)},
$$
where the sum on the right is indexed by conjugacy classes of finite
cyclic subgroups $\Gamma$ in $\op{SL}_2(\mathcal{O}_{K,S})$. 

\begin{remark}
Alternatively, the Farrell--Tate cohomology of $\op{SL}_2(\mathcal{O}_{K,S})$ can be computed using the isotropy spectral sequence for the action of $\op{SL}_2(\mathcal{O}_{K,S})$ on the associated symmetric space $\mathfrak{X}_{K,S}$. It is possible to show that the spectral sequence provides the same direct sum decomposition of Farrell--Tate cohomology as Brown's formula above, because the $\ell$-torsion subcomplex of $\mathfrak{X}_{K,S}$ is a disjoint union of classifying spaces for proper actions of the normalizers of cyclic $\ell$-subgroup of $\op{SL}_2(\mathcal{O}_{K,S})$.  
\end{remark}

\subsection{Cyclic subgrups, their centralizers and normalizers}
To prove Theorem~\ref{thm:main1}, we need to describe the conjugacy classes of finite cyclic $\ell$-subgroups $\Gamma$ in $\op{SL}_2(\mathcal{O}_{K,S})$ and compute their centralizers and normalizers. This is done in Sections~\ref{sec:conjugacy1} and \ref{sec:conjugacy2}. 

The classification proceeds by first setting up a bijection between conjugacy classes of elements of order $\ell$ with characteristic polynomial $T^2-(\zeta_\ell+\zeta_\ell^{-1})T+1$ and classes of ``oriented relative ideals'', cf. Definition~\ref{def:oriented}. This bijection is essentially based on writing out representing matrices for multiplication with $\zeta_\ell$ on suitable $\mathcal{O}_{K,S}$-lattices in $K^2\cong K(\zeta_\ell)$, cf. Proposition~\ref{prop:biject}, and mostly follows classical arguments as in \cite{latimer:macduffee}. Under suitable regularity assumptions, namely (R1) and (R2) in Theorem~\ref{thm:main1}, the set of oriented relative ideals can then be described in terms of relative class groups and the cokernel of norm maps on unit groups, cf. Section~\ref{sec:conjugacy1} for more details. 

Centralizers and normalizers of subgroups generated by elements of finite order can then also be determined very precisely. The main point is that the theory of algebraic groups puts strict constraints on the possible shape of centralizers and normalizers for arithmetic subgroups of algebraic groups: the centralizers turn out to be groups of norm-one units, and the normalizers are either equal to the centralizers or $\mathbb{Z}/2\mathbb{Z}$-extensions thereof. The precise results are proved in Section~\ref{sec:conjugacy2}. The main result formulating the classification of finite cyclic subgroups in $\op{SL}_2(\mathcal{O}_{K,S})$ is Theorem~\ref{thm:conjsl2}. 

Having the description of subgroups of order $\ell$ and their centralizers and normalizers, Parts (i) and (ii) of Theorem~\ref{thm:main1} follow directly from Theorem~\ref{thm:conjsl2} and Brown's formula mentioned above.

\subsection{Farrell--Tate cohomology of centralizers and normalizers}
The description of the relevant normalizers for parts (iii) and (iv) of
Theorem~\ref{thm:main1} can also be found in Theorem~\ref{thm:conjsl2}
resp. Proposition~\ref{prop:aut}. To prove the cohomology statements in parts (iii) and (iv) of Theorem~\ref{thm:main1}, we provide the computation of  the Farrell--Tate cohomology of the normalizers in the two propositions below.

First, we consider the case of cyclic subgroups whose conjugacy class is not Galois-invariant. In this case, the normalizer equals the centralizer and is isomorphic to the kernel of the norm map on units. As a group, the normalizer is then of the form $\mathbb{Z}/n\mathbb{Z}\times\mathbb{Z}^r$ for suitable $n$ and $r$. The computation of the relevant Farrell--Tate cohomology is straightforward:

\begin{proposition} \label{6.1}
Let $A=\mathbb{Z}/n\mathbb{Z}\times \mathbb{Z}^r$, and let $\ell$ be an odd prime with $\ell\mid n$. Then, with $b_1,x_1,\dots,x_r$ denoting classes in degree $1$ and $a_2$ a class of degree $2$, we have 
$$
\widehat{\op{H}}^{\bullet}(A,\mathbb{F}_\ell)\cong \widehat{\op{H}}^\bullet(\mathbb{Z}/n\mathbb{Z},\mathbb{F}_\ell) \otimes_{\mathbb{F}_\ell}\bigwedge^\bullet \mathbb{F}_\ell^r\cong \mathbb{F}_\ell[a_2,a_2^{-1}](b_1,x_1,\dots,x_r).
$$
\end{proposition}

\begin{proof}
We begin with a computation of group cohomology. In this case, the K\"unneth formula implies 
$$
\op{H}^{\bullet}(A,\mathbb{F}_\ell)\cong \op{H}^\bullet(\mathbb{Z}/n\mathbb{Z},\mathbb{F}_\ell)
\otimes_{\mathbb{F}_\ell}\op{H}^\bullet(\mathbb{Z}^r,\mathbb{F}_\ell). 
$$
But $\op{H}^\bullet(\mathbb{Z}^r,\mathbb{F}_\ell)\cong \bigwedge^\bullet\mathbb{F}_\ell^r$, again by iterated application of the K\"unneth formula. Therefore, the first isomorphism claimed above is true with group cohomology instead of Farrell--Tate cohomology. 

Now group cohomology and Farrell--Tate cohomology agree above the virtual cohomological dimension, which in this case is $r$. Moreover, the only finite subgroup of $A$ is $\mathbb{Z}/n\mathbb{Z}$, and hence by \cite{brown:book}*{theorem X.6.7} the group $A$ has periodic Farrell--Tate cohomology. The latter in particular means that there is an integer $d$ such that $\widehat{\op{H}}^i(A,\mathbb{F}_\ell)\cong \widehat{\op{H}}^{i+d}(A,\mathbb{F}_\ell)$ for all $i$. These two assertions imply that the formula we obtained for group cohomology above is also true for Farrell--Tate cohomology. 

The second isomorphism then just combines the first isomorphism with the formula for the cyclic groups discussed earlier in Section~\ref{sec:pregrp}.
\end{proof}

Secondly, we discuss the case where the cyclic subgroup is Galois-invariant. Recall from Theorem~\ref{thm:conjsl2} that the group structure in this case is a semi-direct product $\ker\op{Nm}_1\rtimes \mathbb{Z}/2\mathbb{Z}$, where the action of $\mathbb{Z}/2\mathbb{Z}\cong\op{Gal}(K(\zeta_\ell)/K)$ on $\ker\op{Nm}_1$ is the natural Galois-action.  

\begin{proposition}
The Hochschild--Serre spectral sequence associated to the semi-direct product $\ker\op{Nm}_1\rtimes\mathbb{Z}/2\mathbb{Z}$ degenerates and yields an isomorphism
$$
\widehat{\op{H}}^\bullet(\ker\op{Nm}_1\rtimes\mathbb{Z}/2\mathbb{Z}, \mathbb{F}_\ell)\cong   \widehat{\op{H}}^\bullet(\ker\op{Nm}_1,\mathbb{F}_\ell)^{\mathbb{Z}/2\mathbb{Z}}.
$$
The $\op{Gal}(K(\zeta_\ell)/K)$-action on $\ker\op{Nm}_1$ is by multiplication with $-1$. The invariant classes are then given by $a_2^{\otimes 2i}$ tensor the even exterior powers plus $a_2^{\otimes (2i+1)}$ tensor the odd exterior powers. 
\end{proposition}

\begin{remark}
The statements above provide an $\op{SL}_2$-analogue of the computations of
Anton \cite{anton}. 
\end{remark}

\section{Conjugacy classification of elements of finite order}
\label{sec:conjugacy1}

In this section, we will discuss the conjugacy classification of elements of finite order in $\op{SL}_2$  over rings of $S$-integers in global fields, with some necessary augmentations. In the next section, we will use these results to provide a conjugacy classification of finite cyclic subgroups in $\op{SL}_2$. For a number field $K$, the general conjugacy classification of finite subgroups of $\op{SL}_2(\mathcal{O}_K)$ is due to Kr\"amer \cite{kraemer:diplom}. Special cases for totally real fields appeared before in the study of Hilbert modular  groups, cf. \cite{prestel} and \cite{schneider:75}. A more recent account of this can be found in \cite{maclachlan:torsion}. 

We will discuss here a generalization of these results to rings of $S$-integers in number fields. Some of the necessary modifications for this have been considered in \cite{busch:conjugacy}. Our exposition can be seen as a geometric formulation of the classification result of Latimer-MacDuffee \cite{latimer:macduffee}. 

\subsection{Notation}
\label{sec:notation}

Throughout this section, we let $K$ be a global field, 
$S$  be a non-empty set of places containing the infinite ones and consider the ring of $S$-integers $\mathcal{O}_{K,S}$ in $K$. 
We fix an odd prime $\ell$ different from the  characteristic of $K$, and assume that  $\zeta_\ell+\zeta_\ell^{-1}\in K$. 

Any element of exact order $\ell$ in $\op{SL}_2(\mathcal{O}_{K,S})$ will have characteristic polynomial  of the form $\Psi_\ell(T)=T^2-(\zeta_\ell+\zeta_\ell^{-1})T+1$ for some choice $\zeta_\ell$ of primitive $\ell$-th root of unity. We denote by $K(\Psi_\ell)=K[T]/\Psi_\ell(T)$ the corresponding $K$-algebra,  and by $R_{K,S,\ell}=\mathcal{O}_{K,S}[T]/(\Psi_\ell(T))$ the corresponding order in $K(\Psi_\ell)$. Note that $K(\Psi_\ell)$ is a field if $\zeta_\ell\not\in K$, and is isomorphic to $K\times K$ if $\zeta_\ell\in K$. 

There is an involution $\iota$ of $K(\Psi_\ell)$ given by sending $\zeta_\ell\mapsto \zeta_\ell^{-1}$, i.e., exchanging the two roots of the polynomial $\Psi_\ell(T)$. If $K(\Psi_\ell)$ is a field, then $\iota$ generates $\op{Gal}(K(\Psi_\ell)/K)$; if $K(\Psi_\ell)\cong K\times K$, it exchanges the two factors. In either case, $K(\Psi_\ell)^\iota$ is equal to $K$ embedded as constant polynomials. The involution $\iota$ restricts to $R_{K,S,\ell}$, because $\zeta_\ell$ is integral, and $R_{K,S,\ell}^\iota=\mathcal{O}_{K,S}$. Due to these statements, we will abuse notation and use $\op{Gal}(K(\Psi_\ell)/K)\cong \mathbb{Z}/2\mathbb{Z}\cong \langle\iota\rangle$ even if $K(\Psi_\ell)$ is not a field.

For definitions of class groups, structure of unit groups and other fundamental statements from algebraic number theory, we refer to \cite{neukirch:azt}. We will denote the class group of an $S$-integer ring by $\op{Pic}(\mathcal{O}_{K,S})$. 
Recall that for a finite extension $A\to B$ of commutative rings, there are transfer maps on algebraic K-theory $K_\bullet(B)\to K_\bullet(A)$. Two special cases of such transfer maps will be interesting for us. On the one hand, specializing to rings of Krull dimension one, there is a norm map on class groups $\op{Nm}_0({B/A}):\op{Pic}(B)\to \op{Pic}(A)$. The kernel of $\op{Nm}_0({B/A})$ is called the relative class group $\op{Pic}(B/A)$ of the extension $B/A$. On the other hand, there is a norm map on unit groups $\op{Nm}_1(B/A):B^\times\to A^\times$.

\subsection{Structure of relevant rings}

We first describe the structure of the rings $R_{K,S,\ell}$. If $K(\Psi_\ell)$ is a field, then $R_{K,S,\ell}\cong\mathcal{O}_{K,S}[\zeta_\ell]$  is an order in $K(\Psi_\ell)$. 
In general, it does not need to be a maximal order, hence it may fail to be a Dedekind ring.  

\begin{proposition}
\label{prop:intbasis}
Assume $K(\Psi_\ell)$ is a field, and assume that the prime $(\zeta_\ell-\zeta_\ell^{-1})$ is unramified in the extension $\mathcal{O}_{K,S}/\mathbb{Z}[\zeta_\ell+\zeta_\ell^{-1}]$. Then $\{1,\zeta_\ell\}$ is a  relative integral base for $\mathcal{O}_{K(\Psi_\ell),\tilde{S}}$ over $\mathcal{O}_{K,S}$, where $\tilde{S}$ is the set of places of $K(\Psi_\ell)$ lying over the places in $S$. In particular, $R_{K,S,\ell}\cong\mathcal{O}_{K(\Psi_\ell),\tilde{S}}$ is a Dedekind domain.
\end{proposition}

\begin{proof}
We consider an element $a+b\zeta_\ell\in K(\Psi_\ell)$. We have the trace and norm of this element given by
$$
\op{Tr}(a+b\zeta_\ell)=2a+b(\zeta_\ell+\zeta_\ell^{-1}),\qquad
\op{N}(a+b\zeta_\ell)=a^2+(\zeta_\ell+\zeta_\ell^{-1})ab+b^2.
$$
The element $a+b\zeta_\ell$ is integral over $\mathcal{O}_{K,S}$ if and only if norm and trace are elements of $\mathcal{O}_{K,S}$.

The discriminant of the basis $(1,\zeta_\ell)$ is concentrated at $\ell$, so we only need to worry about divisibility by elements in primes over $\ell$. 
We need a case distinction. 

(1) If $\ell>3$, we have $\zeta_\ell+\zeta_\ell^{-1}\not\equiv 2\mod 3$. Assume norm and trace are divisible by $(\zeta_\ell-\zeta_\ell^{-1})$. Then we consider the reduction $\mathcal{O}_{K,S}/(\zeta_\ell-\zeta_\ell^{-1})$. By assumption, this is a smooth algebra over  $\mathbb{Z}[\zeta_\ell+\zeta_\ell^{-1}]/(\zeta_\ell-\zeta_\ell^{-1})\cong\mathbb{F}_\ell$, in particular it has no nilpotent elements. We compute the universal example of an $\mathbb{F}_\ell$-algebra in which $\op{Tr}=\op{N}=0$; the result is $\mathbb{F}_\ell[A,B]/(A^2,B^2,AB)$. In particular, the elements $A$ and $B$ have to be nilpotent modulo $(\zeta_\ell-\zeta_\ell^{-1})$, hence necessarily have to be zero. Hence we find $(\zeta_\ell-\zeta_\ell^{-1})$ divides $A$ and $B$. Inductively, this deals with divisibility by powers of $(\zeta_\ell-\zeta_\ell^{-1})$. Now the argument for elements which have non-zero reduction in $\mathcal{O}_{K,S}/(\zeta_\ell-\zeta_\ell^{-1})$ is done in the same way: the latter, as an $\mathbb{F}_\ell$-algebra is 
a product of field extensions, hence any quotient of it will again have no nontrivial nilpotent elements. In particular, integrality of the norm and trace implies that the element is already an $\mathcal{O}_{K,S}$-linear combination of $1$ and $\zeta_\ell$. 

(2) The argument for (1) does not work in case $\ell=3$. In this case $2\equiv\zeta_3+\zeta_3^{-1}\mod (3)$. In particular, the norm condition mod 3  is satisfied whenever the trace condition is satisfied mod 3. However, in this case, we have $K/\mathbb{Q}$ has discriminant coprime to $3$ and $\mathbb{Q}(\zeta_3)/\mathbb{Q}$ has discriminant a power of $3$. Then the product of the integral bases for $K$ and $\mathbb{Q}(\zeta_3)$ is an integral basis for the composite $K(\zeta_3)$. In particular, any integral element of $K(\zeta_3)$ is an $\mathcal{O}_{K,S}$-linear combination of $1$ and $\zeta_3$, proving the claim.
\end{proof}

\begin{example}
\label{ex:noninv}
A simple example where $R_{K,S,\ell}$ fails to be a maximal order, due to ramification over $\ell$ is given as follows: take $\ell=3$, let $K=\mathbb{Q}(\sqrt{3})$ and let $S$  contain only the infinite places. In this case, $\mathcal{O}_K=\mathbb{Z}[\sqrt{3}]$ and $R_{K,S,\ell}=\mathbb{Z}[\sqrt{3},\zeta_3]$. The element 
$$
i=\frac{\sqrt{3}+2\sqrt{3}\zeta_3}{3}
$$
is integral, but it is not a $\mathbb{Z}[\sqrt{3}]$-linear combination of $1$ and $\zeta_3$. This type of problems necessitates the discussion of non-invertible ideals.
\end{example}

\begin{corollary}
\label{cor:steinitz1}
Assume $K(\Psi_\ell)$ is a field, and assume that the prime $(\zeta_\ell-\zeta_\ell^{-1})$ is unramified in the extension $\mathcal{O}_{K,S}/\mathbb{Z}[\zeta_\ell+\zeta_\ell^{-1}]$. Then $R_{K,S,\ell}$ is a free $\mathcal{O}_{K,S}$-module of rank two. In particular, there exists an $R_{K,S,\ell}$-ideal with $\mathcal{O}_{K,S}$-basis. 
\end{corollary}

Next, we consider the case where $K(\Psi_\ell)\cong K\times K$ is not a field. In this case, $R_{K,S,\ell}$ is not a Dedekind  domain because it fails to be a domain. The structure of $R_{K,S,\ell}$ is described by the following: 

\begin{proposition}
\label{prop:split}
Assume $K(\Psi_\ell)\cong K\times K$ is not a field. Then the ring $$R_{K,S,\ell}\cong\mathcal{O}_{K,S}[T]/((T-\zeta_\ell)(T-\zeta_\ell^{-1}))$$ sits in a fiber product square
$$
\xymatrix{
  R_{K,S,\ell} \ar[rr]^{\op{ev}_{\zeta_\ell}} \ar[d]_{\op{ev}_{\zeta_\ell^{-1}}} && \mathcal{O}_{K,S} \ar[d]^{\pi}
\\
\mathcal{O}_{K,S} \ar[rr]_{\pi} && \mathcal{O}_{K,S}/(\zeta_\ell-\zeta_\ell^{-1}).
}
$$
The maps denoted by $\op{ev}_x$ evaluate $T\mapsto x$, and $\pi:\mathcal{O}_{K,S}\mapsto\mathcal{O}_{K,S}/(\zeta_\ell-\zeta_\ell^{-1})$ is the natural reduction modulo $(\zeta_\ell-\zeta_\ell^{-1})$. 

As a consequence, we have a ring isomorphism $R_{K,S,\ell}\cong\mathcal{O}_{K,S}\times\mathcal{O}_{K,S}$ if $S_\ell\subseteq S$, i.e., $S$ contains the places lying over $\ell$. By convention, this includes the case where $K$ is a global function field of characteristic different from $\ell$.
\end{proposition}

\begin{proof}
The fiber product square is a special case of the following fiber product for two ideals $I,J$ in ring $R$:
$$
\xymatrix{
  R/(I\cap J)\ar[r] \ar[d] & R/I\ar[d] \\
  R/J \ar[r] & R/(I+J),
}
$$
specialized to $R=\mathcal{O}_{K,S}[T]$, $I=(T-\zeta_\ell)$ and $J=(T-\zeta_\ell^{-1})$. In this case, $I\cap J=(\Psi_\ell(T))$ and $I+J=(\zeta_\ell-\zeta_\ell^{-1},T-\zeta_\ell)$. 

Note that in $\mathbb{Z}[\zeta_\ell]$, we have $(\ell)=(\zeta_\ell^i-1)^{\ell-1}$ for any $0<i<\ell$. In particular, $\zeta_\ell-\zeta_\ell^{-1}$ generates the maximal ideal of $\mathbb{Z}[\zeta_\ell]$ lying over $(\ell)$. If $S_\ell\subseteq S$, then $\zeta_\ell-\zeta_\ell^{-1}$ is invertible, and the quotient ring $\mathcal{O}_{K,S}/(\zeta_\ell-\zeta_\ell^{-1})$ is trivial. In that case, the fiber product diagram shows that $R_{K,S,\ell}$ is the direct product of two copies of $\mathcal{O}_{K,S}$. 
\end{proof}

We can also note that as $\mathcal{O}_{K,S}$-module, $R_{K,S,\ell}$ is always free of rank two, independent of the assumption $S_\ell\subseteq S$.

\begin{corollary}
\label{cor:steinitz2}
The following map is an isomorphism of $\mathcal{O}_{K,S}$-modules:
$$
\mathcal{O}_{K,S}\oplus \mathcal{O}_{K,S}\to
\mathcal{O}_{K,S}\times_{\mathcal{O}_{K,S}/(\zeta_\ell-\zeta_\ell^{-1})}\mathcal{O}_{K,S}: (a,b)\mapsto (a,a+b(\zeta_\ell-\zeta_\ell^{-1})).
$$
In particular, there always exists an ideal in $R_{K,S,\ell}$ which has an $\mathcal{O}_{K,S}$-basis.
\end{corollary} 


\begin{remark}
\label{rem:noninv2}
Since $R_{K,S,\ell}$ fails to be a domain, we cannot really speak about fractional ideal or invertibility of ideals anyway. However, behaviour of ideals in the present situation is fairly similar to the case where $K(\Psi_\ell)$ is a field, and $R_{K,S,\ell}\subseteq K(\Psi_\ell)$ an order. 
We can define a conductor 
$$
\mathfrak{c}=\left\{x\in K\times K\mid x\cdot (\mathcal{O}_{K,S}\times\mathcal{O}_{K,S})\subseteq R_{K,S,\ell}\right\}=
\{(a,b)\subseteq R_{K,S,\ell}\mid a,b\in(\zeta_\ell-\zeta_\ell^{-1})\mathcal{O}_{K,S}\}.
$$
As an example of what can go wrong with ideals which are not coprime to the conductor, we will discuss the case $K=\mathbb{Z}$ (deviating for once from the convention that $\ell$ is an odd prime). In this case, the ring $R_{K,S,\ell}=\mathbb{Z}\times_{\mathbb{Z}/2\mathbb{Z}}\mathbb{Z}$, and the conductor is the ideal $\{(a,b)\in\mathbb{Z}^2\mid a,b\in 2\mathbb{Z}\}$. This ideal is in fact not a principal ideal, it is generated by $(2,0)$ and $(0,2)$. Note that since the idempotents $(1,0)$ and $(0,1)$ in $\mathbb{Z}^2$ fail to lie in $R_{K,S,\ell}$, the ideal generated by $(2,2)$ will contain $(4,0)$ and $(0,4)$,  but not $(2,0)$ or $(0,2)$. For the conductor $\mathfrak{c}$, we have $\mathfrak{c}^2=(2,2)\mathfrak{c}$, and the conductor is the only non-invertible (and non-principal) ideal class. Note also that the conductor, while not a projective $R_{K,S,\ell}$-module, is free as $\mathbb{Z}$-module. This is very similar to the situation of the order $\mathbb{Z}[\sqrt{-3}]$ in $\mathbb{Z}[\zeta_3]$ from Example~\ref{ex:noninv}. 
\end{remark}

\subsection{Conjugacy classes of elements and oriented ideal classes}

As a next step, we can relate conjugacy classes of elements with oriented relative ideal classes. Fix a global field $K$, a set places $S$ and a prime $\ell$ as in Subsection~\ref{sec:notation}. We fix a primitive $\ell$-th root of unity $\zeta_\ell$ and denote by $\mathcal{C}_{K,S,\ell}$ the set of conjugacy classes of order $\ell$ elements in $\op{SL}_2(\mathcal{O}_{K,S})$ whose characteristic polynomial is $\Psi_\ell(T)=T^2-(\zeta_\ell+\zeta_\ell^{-1})T+1$. We will omit $K$ and $S$ whenever they are clear from the context. 

To relate these conjugacy classes of elements to data of a more number-theoretic nature, we consider the following ``oriented relative ideal classes'', i.e.,  classes of ideals whose determinant is principal, plus the additional datum of an element generating the determinant. Definitions like the following have been used in the conjugacy classification of finite order elements in symplectic groups over principal ideal domains, cf. \cite{busch:conjugacy}*{Section 3}. 

\begin{remark}
Note that requiring trivial norm (as in \cite{busch:conjugacy}) is different from requiring trivial determinant (as in the definitions and results below). However, we will only consider the situations in Corollaries~\ref{cor:steinitz1} and \ref{cor:steinitz2} in which we know that the Steinitz class is trivial and therefore norm and determinant are equivalent ideals.
\end{remark}

\begin{definition}
\label{def:oriented}
With the notation above, an oriented relative ideal of $R_{K,S,\ell}$ is a pair $(\mathfrak{a},a)$, where $\mathfrak{a}\subseteq K(\Psi_\ell)$ is a fractional $R_{K,S,\ell}$-ideal together with a choice of generator $a\in\bigwedge^2_{\mathcal{O}_{K,S}}\mathfrak{a}$. This in particular implies that $\mathfrak{a}\cong\mathcal{O}_{K,S}^2$ as $\mathcal{O}_{K,S}$-modules.

Define the following equivalence relation on oriented relative ideals:
$$
(\mathfrak{a},a)\sim (\mathfrak{b},b)\iff
\exists \tau\in K(\Psi_\ell)^\times, \tau\mathfrak{a}=\mathfrak{b}, 
(\op{m}_\tau\wedge \op{m}_\tau)(a)=b,
$$
where $\op{m}_\tau$ denotes multiplication with $\tau$ on $\mathfrak{a}$. Note that $(\op{m}_\tau\wedge\op{m}_\tau)$ is multiplication with the norm $N_{K(\Psi_\ell)/K}(\tau)$ on $\bigwedge^2_{K}K(\Psi_\ell)$ and maps $\bigwedge^2\mathfrak{a}$ to $\bigwedge^2\mathfrak{b}$. 
The $\sim$-equivalence class of an oriented relative ideal $(\mathfrak{a},a)$ is denoted by $[\mathfrak{a},a]$. 

The set of oriented relative ideal classes $\widetilde{\op{Pic}}(R_{K,S,\ell}/\mathcal{O}_{K,S})$ is defined as the set of $\sim$-equivalence classes of oriented relative ideals of $R_{K,S,\ell}$.
\end{definition}

\begin{remark}
We shortly comment on the choice of terminology. 

Including the qualifier ``relative'' is rather common terminology; the relative class group is the kernel of the norm map $\op{Nm}:\op{Pic}(B)\to\op{Pic}(A)$ for a finite extension $B/A$ of Dedekind rings. In the cases we consider, norm and determinant are equivalent ideals.

Using the qualifier ``oriented'' is inspired from a more geometric way of stating the above definition. In a function field situation, we would have a (possibly branched) degree $2$ covering $f:X\to Y$. A line bundle $\mathcal{L}$ on $X$ gives rise to a rank $2$ vector bundle $f_\ast\mathcal{L}$ on $Y$. The requirement in the above definition means that  $f_\ast\mathcal{L}$ has trivial determinant, and includes the datum of a given trivialization. This means that $f_\ast\mathcal{L}$ is in fact an orientable vector bundle, with a given choice of orientation. This puts the class group considered above in  the context of algebraic cohomology theories like oriented Chow groups.
\end{remark}

With the above terminology, we can now establish the following bijection between conjugacy classes of elements and oriented relative ideal classes. The arguments follow those of \cite{busch:conjugacy}.

\begin{proposition}
\label{prop:biject}
There is a bijection between $\widetilde{\op{Pic}}(R_{K,S,\ell}/\mathcal{O}_{K,S})$ and $\mathcal{C}_{K,S,\ell}$.
\end{proposition}

\begin{proof}
Part (1) and (2) of the proof set up the maps between these sets, and part (3) shows that these maps are inverses of each other.

(1) We first describe the map from elements of order $\ell$ with characteristic polynomial $\Psi_\ell(T)$ to oriented ideals. Let $\rho\in\op{SL}_2(\mathcal{O}_{K,S})$ be an element with characteristic polynomial $\Psi_\ell(T)$. Via the standard representation of $\op{SL}_2$, it acts on $K^2$, giving the latter the structure of a rank one $K(\Psi_\ell)$-module. Choosing an isomorphism $K^2\cong K(\Psi_\ell)$, the standard lattice $\mathcal{O}_{K,S}^2\subseteq K^2$ gives rise to a finitely generated $R_{K,S,\ell}$-submodule $\mathcal{O}_{K,S}^2\subseteq K^2\cong K(\Psi_\ell)$, hence a fractional ideal $\mathfrak{a}$. Moreover, in the above, we have chosen a basis of $R_{K,S,\ell}$ as $\mathcal{O}_{K,S}$-module, and this basis gives a generator $a\in\bigwedge_{\mathcal{O}_{K,S}}^2\mathfrak{a}$. 

(1a) The assignment in (1) is independent of the choice of isomorphism $K^2\cong K(\Psi_\ell)$, up to $\sim$-equivalence of oriented relative ideals. Any other isomorphism will be obtained by scaling with  $\lambda\in K(\Psi_\ell)^\times$. This changes the fractional ideal $\mathfrak{a}$ by multiplication with $\lambda$, and the volume form $a$ by multiplication with the norm $N_{K(\Psi_\ell)/K}(\lambda)$. Hence it does not change the class of the associated oriented relative ideal $[\mathfrak{a},a]$ in $\widetilde{\op{Pic}}(R_{K,S,\ell}/\mathcal{O}_{K,S})$.

(1b) More generally, $\op{SL}_2(\mathcal{O}_{K,S})$-conjugate elements are mapped to $\sim$-equivalent oriented relative ideals. Let $\rho\in\op{SL}_2(\mathcal{O}_{K,S})$ be an element with characteristic polynomial $\Psi_\ell(T)$, and let $A\in \op{SL}_2(\mathcal{O}_{K,S})$, so that $\rho$ and $\rho'=A^{-1}\rho A$ are conjugate via $A$. Choose an isomorphism $\phi:K^2\to K(\Psi_\ell)$ which maps $\rho$ to the oriented relative ideal class $[\phi(\mathcal{O}_{K,S}^2),\phi(e_1)\wedge\phi(e_2)]$. The isomorphism $\phi'=\phi\circ A:K^2\to K(\Psi_\ell)$ maps $\rho'$ to the oriented relative ideal class $[\phi'(\mathcal{O}_{K,S}^2),\phi'(e_1)\wedge\phi'(e_2)]$. We claim that these ideal classes are equal. Since $\phi$ is $\mathcal{O}_{K,S}$-linear and $\det(A)=1$, we find that these ideal classes are equal:
\begin{eqnarray*}
[\phi'(\mathcal{O}_{K,S}^2),\phi'(e_1)\wedge\phi'(e_2)]&=&
[\phi(A\cdot\mathcal{O}_{K,S}^2),\phi(A\cdot e_1)\wedge\phi(A\cdot e_2)]\\&=&
[\phi(\mathcal{O}_{K,S}^2),\det(A)\cdot\left(\phi(e_1)\wedge\phi(e_2)\right)]\\
&=&[\phi(\mathcal{O}_{K,S}^2),\phi(e_1)\wedge\phi(e_2)].
\end{eqnarray*}
Essentially, conjugation changes the $K(\Psi_\ell)$-structure of $K^2$, but it leaves invariant the standard lattice $\op{SL}_2(\mathcal{O}_{K,S})$ and the volume form.

(2) Now we describe the map from oriented relative ideals to elements of order $\ell$ with characteristic polynomial $\Psi_\ell(T)$.  Let $(\mathfrak{a},a)$ be an oriented relative ideal. Then $\mathfrak{a}\subseteq K(\Psi_\ell)$ is a finitely generated $R_{K,S,\ell}$-submodule which is isomorphic to $\mathcal{O}_{K,S}^2$, because we have given an explicit trivialization $a\in\bigwedge^2_{\mathcal{O}_{K,S}}\mathfrak{a}$.  Choosing an $\mathcal{O}_{K,S}$-basis of $\mathfrak{a}$ with volume form $a$, we can write multiplication with $\zeta_\ell\in R_{K,S,\ell}$ as a matrix in $\op{SL}_2(\mathcal{O}_{K,S})$. 

(2a) Any two bases of $\mathfrak{a}$ with given volume $a\in\bigwedge^2_{\mathcal{O}_{K,S}}\mathfrak{a}$ are $\op{SL}_2(\mathcal{O}_{K,S})$-conjugate. Therefore,  the class in $\mathcal{C}_{K,S,\ell}$ of the matrix  $\rho\in\op{SL}_2(\mathcal{O}_{K,S})$  associated to the oriented relative ideal $(\mathfrak{a},a)$ is independent of the choice of basis. 

(2b) Assume that the oriented relative ideals $(\mathfrak{a},a)$ and $(\mathfrak{b},b)$ are $\sim$-equivalent, i.e., there exists $\tau\in K(\Psi_\ell)^\times$ with $\tau\mathfrak{a}=\mathfrak{b}$ and $(\op{m}_\tau\wedge\op{m}_\tau)(a)=b$. In particular, $\op{m}_\tau$ induces an isomorphism $\mathfrak{a}\cong\mathfrak{b}$ as $R_{K,S,\ell}$-modules. An $\mathcal{O}_{K,S}$-basis of $\mathfrak{a}$ will be mapped by $\op{m}_\tau$ to a basis of $\mathfrak{b}$, and the corresponding representing matrices will be $\op{GL}_2(\mathcal{O}_{K,S})$-conjugate. Moreover, since the volume forms correspond under the $R_{K,S,\ell}$-module isomorphism, the conjugating change-of-basis matrix will  already lie in $\op{SL}_2(\mathcal{O}_{K,S})$.

(3) It is now easy to see that the maps constructed in (1) and (2) are inverses of each other. From a matrix, we get a fractional $R_{K,S,\ell}$-ideal $\mathfrak{a}\subseteq K(\Psi_\ell)$, with chosen basis in which $\Psi_\ell$ acts via the given matrix. Conversely, starting from an oriented relative ideal, we can choose an $\mathcal{O}_{K,S}$-basis and write out the representing matrix, which will then give back the $R_{K,S,\ell}$-module structure we started with.
\end{proof}

\begin{remark}
Our main interest in this paper is the conjugacy classification of finite order subgroups in $\op{SL}_2(\mathcal{O}_{K,S})$. However, there are similar results in the case $\op{GL}_2$, which actually can be formulated slightly easier. The basic correspondence between ideal classes in extensions and conjugacy classes of  elements  was already described by Latimer--MacDuffee \cite{latimer:macduffee}, and later generalized by Taussky and Bender. A variation of \cite{bender}*{theorem~1} shows that for a Dedekind domain $R$, there is a bijective correspondence between conjugacy classes of elements of order $\ell$ and ideal classes in $R[\zeta_\ell]$ which have an $R$-basis. 

The actual correspondence is fairly easy to setup, and is well explained in \cite{latimer:macduffee}, or \cite{conrad} for a more modern exposition. An ideal of $R[\zeta_\ell]$ with $R$-basis gives rise to a conjugacy class of  elements in $\op{GL}_2(R)$ by writing multiplication with $\zeta_\ell$ in some $R$-basis. Conversely, given an element $\rho\in\op{GL}_2(R)$, its action on $R^2$ gives the latter the structure of fractional ideal for $R[\zeta_\ell]$.  The arguments are similar to what we have done for $\op{SL}_2$ above, only easier because there is no fixed orientation to carry around.
\end{remark}

\newpage
\begin{proposition}
Assume one of the following conditions: 
\begin{enumerate}[({R}1)]
\item $\ell\not\in K$ and the prime $(\zeta_\ell-\zeta_\ell^{-1})$ is unramified in the extension $\mathcal{O}_{K,S}/\mathbb{Z}[\zeta_\ell+\zeta_\ell^{-1}]$. 
\item $\ell\in K$ and $S_\ell\subseteq S$.
\end{enumerate}
Then there is a natural group structure on $\widetilde{\op{Pic}}(R_{K,S,\ell}/\mathcal{O}_{K,S})$, given by multiplication of ideals and volume forms.
\end{proposition}

\begin{proof}
Multiplication is given by ideal multiplication and multiplication of volume forms. The natural element is the trivial ideal class oriented by $1$. Here $\bigwedge^2_{\mathcal{O}_{K,S}}R_{K,S,\ell}$ is viewed as the natural sub-$\mathcal{O}_{K,S}$-module of $\bigwedge^2_KK(\Psi_\ell)\cong K$. Since $T\in R_{K,S,\ell}$ has norm $1$, the basis $(1,T)$ maps to $1\wedge T\in \bigwedge^2_{\mathcal{O}_{K,S}}R_{K,S,\ell}$, which maps to $1$ in $\bigwedge^2_KK(\Psi_\ell)\cong K$. This implies that all the axioms for a group operation are satisfied except the invertibility.

In case (R1), Proposition~\ref{prop:intbasis} implies that $R_{K,S,\ell}$ is a Dedekind ring. In particular, ideals are invertible with respect to multiplication, the inverse of $I$ is given by $\iota(I)$. Since the inverse  of an ideal has inverse norm, forming inverses preserves the property of having an $\mathcal{O}_{K,S}$-basis. 

In case (R2), $R_{K,S,\ell}$ is not a Dedekind ring, but it is of the form $\mathcal{O}_{K,S}\times\mathcal{O}_{K,S}$. The ideals are then given by pairs of ideals of $\mathcal{O}_{K,S}$, with entrywise multiplication. Since $\mathcal{O}_{K,S}$ is a Dedekind ring, ideals in $R_{K,S,\ell}$ are invertible for the multiplication. 

In both cases, the volume form of $I$ induces a volume form for $\iota(I)$. Since elements in the image of the norm $N_{R_{K,S,\ell}/\mathcal{O}_{K,S}}$ provide volume forms $\sim$-equivalent to $1$, this provides the inverse. 
\end{proof}

We will speak of the \emph{oriented relative class group of $R_{K,S,\ell}/\mathcal{O}_{K,S}$} in those situations where the assumptions (R1)  or (R2) are satisfied and the above proposition implies that the set $\widetilde{\op{Pic}}(R_{K,S,\ell}/\mathcal{O}_{K,S})$ has a group structure.

\subsection{Oriented class groups in the case $\zeta_\ell\not\in K$}

The next step is to identify the set of oriented relative ideal classes in terms of more conventional data from algebraic number theory. This will be done in the next subsections, handling the two cases $\zeta_\ell\not\in K$ and $\zeta_\ell\in K$ separately. As it turns out, the conjugacy classification of elements of finite order in $\op{SL}_2(\mathcal{O}_{K,S})$ is controlled mostly by kernels and cokernels of norm maps.




\begin{proposition}
\label{prop:lmsl}
Under assumption (R1), the oriented relative class group $\widetilde{\op{Pic}}(R_{K,S,\ell})$ sits in an extension 
$$
1\to \op{coker}\left(\op{Nm}_1:R_{K,S,\ell}^\times\to \mathcal{O}_{K,S}^\times\right)\to \widetilde{\op{Pic}}(R_{K,S,\ell}/\mathcal{O}_{K,S})\to   \op{Pic}(R_{K,S,\ell}/\mathcal{O}_{K,S})\to 1.
$$
\end{proposition}

\begin{proof}
The proof proceeds as in \cite{busch:conjugacy}*{Proposition 3.10}.  

We first note that there is a natural group homomorphism
$$
\widetilde{\op{Pic}}(R_{K,S,\ell}/\mathcal{O}_{K,S})\to
\op{Pic}(R_{K,S,\ell})
$$ 
mapping an oriented relative ideal to its underlying $R_{K,S,\ell}$-ideal. The image consists exactly of the ideal classes in $R_{K,S,\ell}$ which have an $\mathcal{O}_{K,S}$-basis. In particular, under assumption (R1), trivial determinant and trivial norm are the same, cf. Corollary~\ref{cor:steinitz1}, and the map above induces a surjection 
$$
\widetilde{\op{Pic}}(R_{K,S,\ell}/\mathcal{O}_{K,S})\to
\op{Pic}(R_{K,S,\ell}/\mathcal{O}_{K,S}).
$$ 

We also get an injective group homomorphism 
$$
\op{coker}\left(\op{Nm}_1:R_{K,S,\ell}^\times\to
\mathcal{O}_{K,S}^\times\right)\to
\widetilde{\op{Pic}}(R_{K,S,\ell}/\mathcal{O}_{K,S})
$$
by sending  an element $u\in\mathcal{O}_{K,S}^\times$ to the oriented relative  ideal $(R_{K,S,\ell},u\wedge T)$ where $u\wedge T$  is the orientation corresponding to the $\mathcal{O}_{K,S}$-basis $(u,T)$ of $R_{K,S,\ell}$. This map will factor through the quotient $\op{coker}\op{Nm}_1$ and induce an injection as claimed  by the definition  of equivalence of oriented relative  ideals. 

Exactness in the middle also follows directly from the definition of equivalence of oriented relative ideals. 
\end{proof}

\begin{remark}
The natural origin of the short exact sequence in Proposition~\ref{prop:lmsl} 
 is the long exact sequence associated to the fiber sequence of K-theory spectra 
$$\op{hofib}\mathbf{Nm}\to \mathbf{K}(R_{K,S,\ell})\stackrel{\mathbf{Nm}}{\longrightarrow} \mathbf{K}(\mathcal{O}_{K,S})$$
 (up to a discussion getting rid of the $\mathbb{Z}$-summands in $K_0$).
\end{remark}

We have now seen how to relate the number of classes of oriented relative ideals to number-theoretic data: the possible underlying ideals are counted via the relative class group, and the possible orientations are counted via the cokernel of the norm map. It is possible to get even more information on these constituent pieces: for the relative class group, one could use  class number formulas. On the other hand, the cokernel of the norm map on units can also be understood by generalizing the  discussion of \cite{busch:conjugacy}*{section 3.2} to get the following statement.

\begin{proposition}
The cokernel $\op{coker}\op{Nm}_1(R_{K,S,\ell}/\mathcal{O}_{K,S})$ of the norm-map on units is a $\mathbb{Z}/2\mathbb{Z}$-module whose rank equals  the number of inert places of the extension $R_{K,S,\ell}/\mathcal{O}_{K,S}$.
\end{proposition}

\subsection{Oriented class groups in the case  $\zeta_\ell\in K$}

Understanding the oriented class group in the case $\zeta_\ell\in K$ is easier. 

\begin{proposition}
\label{prop:picsplit}
Under assumption (R2), we have an isomorphism
$$
\widetilde{\op{Pic}}(R_{K,S,\ell}/\mathcal{O}_{K,S})\cong \op{Pic}(\mathcal{O}_{K,S}). 
$$
\end{proposition}

\begin{proof}
The most important ingredient is Proposition~\ref{prop:split}, which implies that under assumption (R2), we have $R_{K,S,\ell}\cong\mathcal{O}_{K,S}\times\mathcal{O}_{K,S}$ as rings. 
In particular, ideal classes of $R_{K,S,\ell}$ are in bijection with pairs of ideal classes in $\mathcal{O}_{K,S}$, 
and an ideal has an $\mathcal{O}_{K,S}$-basis if and only if it is equivalent to one of the form $(\mathcal{I},\mathcal{I}^{-1})$.
We therefore get a surjection $\widetilde{\op{Pic}}(R_{K,S,\ell}/\mathcal{O}_{K,S})\to \op{Pic}(\mathcal{O}_{K,S})$. 
To show injectivity, we look at the possible orientations of the trivial ideal class. These are given by generators of the determinant, so they differ by units in $\mathcal{O}_{K,S}$. Now scaling with a unit in the image of the norm map $R_{K,S,\ell}^\times\to\mathcal{O}_{K,S}^\times$ does not change the orientation. However, the norm map is simply the multiplication map $R_{K,S,\ell}^\times\cong(\mathcal{O}_{K,S}^\times)^2\to\mathcal{O}_{K,S}^\times$. Surjectivity of the norm map then implies that all orientations are equivalent, hence we get the isomorphism as claimed.
\end{proof}

\begin{remark}
It is still possible to describe the oriented class group in case (R2) is not satisfied. In this case, we have to restrict to those oriented ideals whose underlying ideals are invertible. The resulting class group is then given by an extension
$$
0\to
\op{coker}\left(\mathcal{O}_{K,S}^\times\to \left(\mathcal{O}_{K,S}/(\zeta_\ell-\zeta_\ell^{-1})\right)^\times\right)\to \op{Pic}(R_{K,S,\ell})\to\op{Pic}(\mathcal{O}_{K,S})\to 0.
$$
This is basically a form of Milnor patching for projective modules in the fiber square of Proposition~\ref{prop:split}, with the cokernel-of-reduction on units classifying the possible gluing data. The exact sequence gives rise to a version of Dedekind's formula for class groups of orders. In this case, there can also be non-trivial orientations coming from inert places over $(\zeta_\ell-\zeta_\ell^{-1})$. However, the most problematic part of understanding oriented relative ideals is the possible non-invertibility of ideals as discussed in  Remark~\ref{rem:noninv2}. This is the reason for staying away from this sort of cases altogether.
\end{remark}

\section{Conjugacy classification of finite cyclic subgroups and descriptions of normalizers}
\label{sec:conjugacy2}

In Section~\ref{sec:conjugacy1}, we recalled the conjugacy classification of elements of order $\ell$ in $\op{SL}_2(\mathcal{O}_{K,S})$ with characteristic polynomial $\Psi_\ell(T)$. What remains to be done is the description of the conjugacy classification of \emph{subgroups} of order $\ell$ and the description of their centralizers and normalizers. 
The difference between the classification of finite order elements and the classification of finite cyclic subgroups is completely controlled by the action of the ``Galois group'' 
$\op{Gal}(K(\Psi_\ell)/K)\cong\mathbb{Z}/2\mathbb{Z}$. Again, the statements for $\op{SL}_2(\mathbb{Z}[1/n])$ can be found in \cite{busch:conjugacy}, and we provide some necessary augmentations to deal with the general case $\op{SL}_2(\mathcal{O}_{K,S})$. 

In this section, we continue to use the notation set up at the beginning of Section~\ref{sec:conjugacy1}. We will denote by $\mathcal{K}_{K,S,\ell}$ the set of conjugacy classes of finite cyclic subgroups of order $\ell$ in $\op{SL}_2(\mathcal{O}_{K,S})$, and write $\mathcal{K}_\ell$ if the number ring is clear from the context. If we have any cyclic subgroup $\Gamma$ of order $\ell$ in $\op{SL}_2(\mathcal{O}_{K,S})$, then for any primitive $\ell$-th root of unity $\zeta$ there will be an element of $\Gamma$ having characteristic polynomial $\Psi_\ell(T)=T^2-(\zeta+\zeta^{-1})T+1$. In particular,  $\mathcal{K}_\ell$ is a quotient of $\mathcal{C}_\ell$, and the difference appears whenever, for an element $g$, the normalizer $N$ of the subgroup $\langle g\rangle$ acts non-trivially on this subgroup. 

\subsection{Centralizers and norm one units}

We first consider the centralizers of elements of finite order. Again, we have to distinguish between the cases where $K(\Psi_\ell)$ is an extension field of $L$ or where $K(\Psi_\ell)\cong K\times K$. 

\begin{proposition}
\label{prop:centralizer1}
Assume $K(\Psi_\ell)$ is a field and that condition (R1) is satisfied. If $\rho$ is an element of order $\ell$ and $M\in\op{SL}_2(\mathcal{O}_{K,S})$ centralizes $\langle\rho\rangle$, then $M$ is given by multiplication by a norm-one unit $u\in\ker\op{Nm}_1\subseteq R_{K,S,\ell}^\times$. In particular, we have 
$$
C_{\op{SL}_2(\mathcal{O}_{K,S})}(\langle\rho\rangle)\cong
\ker\left(R_{K,S,\ell}^\times
\stackrel{\op{Nm}_1}{\longrightarrow}\mathcal{O}_{K,S}^\times\right). 
$$
\end{proposition}

\begin{proof}
Under the correspondence set up in Section~\ref{sec:conjugacy1} (with $\zeta_\ell$ chosen such that $\Psi_\ell(T)=T^2-(\zeta_\ell+\zeta_\ell^{-1})T+1$ is the characteristic polynomial of $\rho$), the element $\rho$ corresponds to an oriented relative ideal $(\mathfrak{a},a)$, where $\mathfrak{a}\subseteq K(\Psi_\ell)$ is an invertible fractional $R_{K,S,\ell}$-ideal and $a\in\bigwedge_{\mathcal{O}_{K,S}}^2\mathfrak{a}$. The element $\rho$ is represented as multiplication by $\zeta_\ell$ on $\mathfrak{a}$. Since the ring $R_{K,S,\ell}$ is generated by $\zeta_\ell$ and elements from $\mathcal{O}_{K,S}$, any matrix $M$ that commutes with multiplication with $\zeta_\ell$ necessarily commutes with multiplication with any element from $R_{K,S,\ell}$. We can then consider the associated algebraic group $\op{GL}_2(K)$ acting on the two-dimensional $K$-vector space $K(\Psi_\ell)$ in which the matrices representing multiplication with elements of $K(\Psi_\ell)$ form a maximal torus. Therefore, we see that any matrix $M\in\op{SL}_2(\mathcal{O}_{K,S})$ which commutes with $\zeta_\ell$ centralizes a maximal torus of $\op{GL}_2(K)$ after embedding $\op{SL}_
2(\mathcal{O}_{K,S})\subseteq\op{GL}_2(K)$. From the theory of algebraic groups, we see that, as an element of $\op{GL}_2(K)$, $M$ must itself be an element of the maximal torus, hence it necessarily is multiplication with a unit. The determinant of multiplication with a unit is given by the norm of the unit. Therefore, if an element $M\in\op{SL}_2(\mathcal{O}_{K,S})$ centralizes $\zeta_\ell$, then it is given by multiplication with a norm-one unit.
\end{proof}

\begin{proposition}
\label{prop:centralizer2}
Assume $K(\Psi_\ell)\cong K\times K$ and that condition (R2) is satisfied. If $\rho$ is an element of order $\ell$ and  $M\in\op{SL}_2(\mathcal{O}_{K,S})$ centralizes  $\langle\rho\rangle$, then $M$ is given by multiplication by a diagonal matrix $(u,u^{-1})$ with $u\in\mathcal{O}_{K,S}^\times$.  In particular, we have 
$$
C_{\op{SL}_2(\mathcal{O}_{K,S})}(\langle\rho\rangle)\cong\mathcal{O}_{K,S}^\times.
$$
\end{proposition}

\begin{proof}
As before, the element $T$ generates $R_{K,S,\ell}$ as $\mathcal{O}_{K,S}$-algebra. Therefore, any matrix $M\in\op{SL}_2(\mathcal{O}_{K,S})$ commuting with multiplication by $T$ will commute with multiplication by any element from $R_{K,S,\ell}$. 
Passing to the algebraic group $\op{SL}_2(K)$, such a matrix $M$ commuting with multiplication by $T$ again commutes with the whole maximal torus of $\op{SL}_2(K)$. Therefore, it must already be a diagonal matrix, hence of the form $\op{diag}(a,a^{-1})$ for $a\in R_{K,S,\ell}^\times$. 
\end{proof}

\subsection{Normalizers, dihedral overgroups and Galois action}

Next, we discuss normalizers of subgroups of finite order in $\op{SL}_2(\mathcal{O}_{K,S})$, as these are relevant data for the computation of Farrell--Tate cohomology. Looking at the proofs of Propositions~\ref{prop:centralizer1} and \ref{prop:centralizer2}, we get the following result: 

\begin{proposition}
\label{prop:aut} 
In the situation of Propositions~\ref{prop:centralizer1} and \ref{prop:centralizer2}, 
the Weyl group of the subgroup $\langle\rho\rangle$ generated by $\rho$ is given by
$$
N_{\op{SL}_2(\mathcal{O}_{K,S})}(\langle\rho\rangle)/C_{\op{SL}_2(\mathcal{O}_{K,S})}(\langle\rho\rangle)\cong 
  \op{Stab}_{\op{Gal}(K(\Psi_\ell)/K)}(\mathcal{I}_\rho),
$$
where  $\mathcal{I}_\rho$ is the ideal class corresponding to the element $\rho$.  
\end{proposition}

\begin{proof}
In each of the cases of Propositions~\ref{prop:centralizer1} resp. \ref{prop:centralizer2}, 
we see that an element centralizing $\rho$ must centralize a maximal torus. 
Similarly, if an element normalizes the subgroup $\langle\rho\rangle$, then it already normalizes a maximal torus of $\op{SL}_2(K)$. 
The theory of algebraic groups tells us that the normalizer of a maximal torus of $\op{SL}_2(K)$ is of the form 
$K^\times\rtimes\mathbb{Z}/2\mathbb{Z}$, with the finite group quotient acting by inversion on $K^\times$. 
In particular, if an element normalizes $\langle\rho\rangle$, but does not leave $\rho$ invariant, 
then it must map $\rho$ to $\rho^{-1}$. 
Then necessarily, we have $\mathcal{I}_\rho\cong\mathcal{I}_\rho^{\sigma}$, i.e., 
the ideal class $\mathcal{I}$ in $\widetilde{\op{Pic}}(R_{K,S,\ell}/\mathcal{O}_{K,S})$ is invariant under the action of 
$\op{Gal}(K(\Psi_\ell)/K)\cong\mathbb{Z}/2\mathbb{Z}$. 
On the other hand, if the Galois action does not leave the ideal class $\mathcal{I}_\rho$ invariant, then 
$\rho$ and $\rho^{-1}$ cannot be conjugate, hence normalizer and centralizer agree in this case.
\end{proof}

\begin{corollary}
\label{cor:dihedral}
In the situation of Proposition~\ref{prop:aut}, a cyclic subgroup
$\Gamma=\langle\rho\rangle$ of order $\ell$ in $\op{SL}_2(\mathcal{O}_{K,S})$ is embeddable in a dihedral subgroup of $\op{SL}_2(\mathcal{O}_{K,S})$ precisely when the associated ideal $\mathcal{I}_\rho$ is invariant under the action of
$\op{Gal}(K(\Psi_\ell)/K)$. 

The conjugacy classes of dihedral overgroups of $\Gamma$ are in bijection with the number of orientations of 
$\mathcal{I}^{-1}\otimes\mathcal{I}^\sigma$, which in turn is in bijection with the group 
$$
\ker\left(\op{Nm}_1:R_{K,S,\ell}^\times\to \mathcal{O}_{K,S}^\times\right)\otimes_{\mathbb{Z}}\mathbb{Z}/2\mathbb{Z}
$$ 
of square residues of norm 1 units.
\end{corollary}

\begin{remark}
The above results on computations of normalizers also explain exactly how to pass from the conjugacy classification for elements (with fixed characteristic polynomial) to the conjugacy classification for subgroups. What can happen is that the normalizer of the subgroup $\langle\rho\rangle$ identifies $\rho$ and $\rho^{-1}$ which have the same characteristic polynomial $T^2-(\zeta_\ell+\zeta_\ell^{-1})T+1$. In particular, the conjugacy classes of  subgroups are given as orbit set of $\mathcal{C}_{K,S,\ell}\cong \widetilde{\op{Pic}}(R_{K,S,\ell}/\mathcal{O}_{K,S})$ under the $\op{Gal}(K(\Psi_\ell)/K)$-action.
\end{remark}

It remains to better understand the Galois action on the oriented relative class group. The following is just a consequence of writing out the definition of oriented ideals:

\vbox{
\begin{proposition}
\label{prop:galois}
\begin{enumerate}
\item Assume $\zeta_\ell\not\in K$ and assumption (R1) is satisfied. Then the action of $\op{Gal}(K(\Psi_\ell)/K)$ on $\widetilde{\op{Pic}}(R_{K,S,\ell}/\mathcal{O}_{K,S})$ is induced from sending an oriented ideal $(\mathfrak{a},a)\mapsto(\iota(\mathfrak{a}),a')$, where $a'$ is given as follows:  represent the orientation $a\in\bigwedge^2_{\mathcal{O}_{K,S}}\mathfrak{a}$ by a corresponding $\mathcal{O}_{K,S}$-basis s.th. $a=x_1\wedge x_2$. Then $a'=\iota(x_1)\wedge\iota(x_2)$, which is a volume form for $\iota(\mathfrak{a})$. 
\item Assume $\zeta_\ell\in K$ and assumption (R2) is satisfied. Then the action of $\op{Gal}(K(\Psi_\ell)/K)$ on $\widetilde{\op{Pic}}(R_{K,S,\ell}/\mathcal{O}_{K,S})\cong \op{Pic}(\mathcal{O}_{K,S})$ is given by the inverse. 
\end{enumerate}
\end{proposition}
}

In particular, the action of the Galois group on orientations is not the trivial one (as one could think from the identification as quotient of $\mathcal{O}_{K,S}^\times$). This implies that the Galois action does not respect the group structure on orientations (viewed as cokernel of $\op{Nm}_1$).

\begin{example}
We discuss the Galois action in the simplest case, namely elements and subgroups of $\op{SL}_2(\mathbb{Z})$ of order~$3$. 
In this case, we have $\mathcal{O}_{K,S}=\mathbb{Z}$, $R_{K,S,\ell}=\mathbb{Z}[\zeta_3]$ and condition (R1) is satisfied. 
The ring $\mathbb{Z}[\zeta_3]$ is euclidean, so there is only a trivial ideal class. 
Moreover, the norm map on units is the trivial map $\mathbb{Z}[\zeta_3]^\times\cong\mu_6\to\mu_2\cong\{1,-1\}$. 
Therefore, we find two conjugacy classes of elements of exact order $3$ in $\op{SL}_2(\mathbb{Z})$, 
given by the two possible orientations of $\mathbb{Z}[\zeta_3]$ as free $\mathbb{Z}$-module of rank two. 
We can choose a $\mathbb{Z}$-basis $(1,\zeta_3)$ for $\mathbb{Z}[\zeta_3]$. 
The Galois group $\op{Gal}(\mathbb{Q}(\zeta_3)/\mathbb{Q})$ will send this basis to $(1,\zeta_3^2)$, 
which is obviously orientation-reversing. 
The result is the obvious one - there is a unique conjugacy class of subgroups of order $3$ in $\op{SL}_2(\mathbb{Z})$. 
\end{example}

\subsection{Explicit formulas}

Finally, we can combine the previous results into a result describing conjugacy classes of finite cyclic subgroups in $\op{SL}_2(\mathcal{O}_{K,S})$ with $K$ a global field, under suitable regularity assumptions.  

\newpage
\begin{theorem}
\label{thm:conjsl2}
Let $K$ be a global field, and fix an odd prime $\ell$ different from the
characteristic of $K$.
\begin{enumerate}
\item Assume $\zeta_\ell\not\in K$ and that assumption (R1) is satisfied. The set $\mathcal{C}_{\ell}$ of conjugacy classes of elements of order $\ell$ with characteristic polynomial $\Psi_\ell(T)=T^2-(\zeta_\ell+\zeta_\ell^{-1})T+1$ is non-empty, has a group structure and sits in the extension 
$$
1\to \op{coker}\left(\op{Nm}_1:R_{K,S,\ell}^\times\to 
  \mathcal{O}_{K,S}^\times\right)\to \mathcal{C}_\ell\to
\ker\left(\op{Nm}_0:\op{Pic}(R_{K,S,\ell})\to
\op{Pic}(\mathcal{O}_{K,S})\right)\to 1.
$$
Additionally, the set $\mathcal{C}_{\ell}$ has an action of ${\op{Gal}(K(\Psi_\ell)/K)}$, 
obtained from its identification with oriented ideals in $R_{K,S,\ell}$. 
Denoting by $\mathcal{K}_{\ell}$ the conjugacy classes  of subgroups of order  $\ell$ in $\op{SL}_2(\mathcal{O}_{K,S})$, we have an isomorphism  $\mathcal{K}_\ell\cong\mathcal{C}_{\ell}/{\op{Gal}(K(\Psi_\ell)/K)}$.

A finite group   $\Gamma$ with $[\Gamma]\in\mathcal{K}_\ell$ is contained
in a dihedral overgroup if and only if the corresponding element $\mathcal{I}(\Gamma)$ in
$\mathcal{C}_{\ell}$ is $\op{Gal}(K(\Psi_\ell)/K)$-invariant.

If the corresponding element $\mathcal{I}(\Gamma)$ in 
$\mathcal{C}_{\ell}$ is $\op{Gal}(K(\Psi_\ell)/K)$-invariant, then the normalizer of $\Gamma$ in   $\op{SL}_2(\mathcal{O}_{K,S})$ is isomorphic to
$$\ker\left(\op{Nm}_1:R_{K,S,\ell}^\times\to\mathcal{O}_{K,S}^\times
\right)\rtimes\mathbb{Z}/2\mathbb{Z}$$ 
with the quotient $\mathbb{Z}/2\mathbb{Z}$ acting via inversion.

If the corresponding element $\mathcal{I}(\Gamma)$ in 
$\mathcal{C}_{\ell}$ is not $\op{Gal}(K(\Psi_\ell)/K)$-invariant, then the normalizer of $\Gamma$ in $\op{SL}_2(\mathcal{O}_{K,S})$ is isomorphic to   $\ker\left(\op{Nm}_1:R_{K,S,\ell}^\times\to\mathcal{O}_{K,S}^\times \right)$.
\item Assume $\zeta_\ell\in K$ and that assumption (R2) is satisfied. Then the conjugacy classes of elements of order $\ell$ with characteristic polynomial $\Psi_\ell(T)=T^2-(\zeta_\ell+\zeta_\ell^{-1})T+1$ are in bijection with $\op{Pic}(\mathcal{O}_{K,S})$. 
Denoting the involution $\mathcal{I}\mapsto \mathcal{I}^{-1}$ by $\iota$, the conjugacy classes $\mathcal{K}_\ell$ of subgroups of order $\ell$ are in bijection with the orbit set $\op{Pic}(\mathcal{O}_{K,S})/\iota$. 

Any such finite group $\Gamma$ is contained in a dihedral overgroup if and only if the corresponding element $\mathcal{I}(\Gamma)$ in $\op{Pic}(\mathcal{O}_{K,S})$ is $\mathbb{Z}/2\mathbb{Z}$-invariant. 

If the corresponding element $\mathcal{I}(\Gamma)$ in 
$\mathcal{C}_{\ell}$ is $\iota$-invariant, then the normalizer of $\Gamma$ in    $\op{SL}_2(\mathcal{O}_{K,S})$ is isomorphic to $\mathcal{O}_{K,S}^\times\rtimes\mathbb{Z}/2\mathbb{Z}$ with $\mathbb{Z}/2\mathbb{Z}$ acting  by inversion. 

If the corresponding element $\mathcal{I}(\Gamma)$ in 
$\mathcal{C}_{\ell}$ is not $\iota$-invariant, then the normalizer of $\Gamma$ in $\op{SL}_2(\mathcal{O}_{K,S})$ is isomorphic to $\mathcal{O}_{K,S}^\times$.
\end{enumerate}
\end{theorem}

\begin{proof}
(i) The conjugacy classification for elements is given by the bijection of Proposition~\ref{prop:biject}. The exact sequence is Proposition~\ref{prop:lmsl}. The description of centralizers and normalizers  is given in Proposition~\ref{prop:centralizer1} and Proposition~\ref{prop:aut}. The description of dihedral overgroups is Corollary~\ref{cor:dihedral}. From the description of normalizers, we get the conjugacy classification of subgroups as claimed.

Part (ii) is proved along the same lines, using Propositions~\ref{prop:picsplit} and \ref{prop:centralizer2} at the suitable places.
\end{proof}

\begin{remark}
See also \cite{sl2parabolic} for similar results concerning
$\op{SL}_2(K[C])$ with $C$ a smooth affine curve over an algebraically
closed field $K$. 
\end{remark}

\begin{remark}
Similar results can be obtained for $\op{GL}_2(\mathcal{O}_{K,S})$ with $K$ a global field, with only small modifications to the proofs. We only formulate the result.
\begin{enumerate}
\item Assume $\zeta_\ell\in K$. Then the conjugacy classes $[\Gamma]$ of $\ell$-order subgroups in $\op{GL}_2(\mathcal{O}_{K,S})$ are in bijection with   elements of $\op{Pic}(\mathcal{O}_{K,S})/\iota$. Such a finite group $\Gamma$ is contained in a dihedral overgroup if and only if the corresponding element $\mathcal{I}(\Gamma)$ in $\op{Pic}(\mathcal{O}_{K,S})$ is $\mathbb{Z}/2\mathbb{Z}$-invariant; in that case, all dihedral overgroups are conjugate. The normalizer of $\Gamma$ in $\op{GL}_2(\mathcal{O}_{K,S})$ is isomorphic   to \mbox{$(\mathcal{O}_{K,S}^\times)^2\rtimes\mathbb{Z}/2\mathbb{Z}$} if there is a dihedral overgroup, and isomorphic to $(\mathcal{O}_{K,S}^\times)^2$ otherwise. 
\item Assume $\zeta_\ell\not\in K$. Then the conjugacy classes $[\Gamma]$ of $\ell$-order subgroups in $\op{GL}_2(\mathcal{O}_{K,S})$ are classified by elements of the orbit set $\op{Pic}(R_{K,S,\ell}/  \mathcal{O}_{K,S})/{\op{Gal}(K(\Psi_\ell)/K)}$. Such a finite group $\Gamma$ is  contained in a dihedral overgroup if and only if the corresponding element $\mathcal{I}(\Gamma)$ in $\op{Pic}(R_{K,S,\ell}/\mathcal{O}_{K,S})$ is   $\mathbb{Z}/2\mathbb{Z}$-invariant; in that case, all dihedral overgroups are conjugate. The normalizer of $\Gamma$ in $\op{GL}_2(\mathcal{O}_{K,S})$ is isomorphic to  $\mathcal{O}_{K,S}[\zeta_\ell]^\times\rtimes\mathbb{Z}/2\mathbb{Z}$ if there is a dihedral overgroup, and to $\mathcal{O}_{K,S}[\zeta_\ell]^\times$ otherwise. 
\end{enumerate}
\end{remark}

\begin{remark}
The above results are more or less immediate generalizations of the classification in \cite{busch:conjugacy}. The relation to the classification results of Prestel \cite{prestel}, Schneider \cite{schneider:75}, Kr\"amer \cite{kraemer:diplom} or Maclachlan \cite{maclachlan:torsion} is a bit more subtle to discuss.  We restrict ourselves to mention the two major differences: one is due to the fact that the cited works consider the more general situation of possibly non-split quaternion algebras (instead of $\op{M}_2(A)$ considered here). The second difference is that the cited works provide much more elaborate formulas for the orders of relative class groups, where we are basically stopping at Proposition~\ref{prop:lmsl}. For our applications to computations of Farrell--Tate cohomology we do not need the actual numbers, but the more conceptual explanations provided by the results  above.
\end{remark}

\section{Application I: Quillen conjecture and non-detection}
\label{sec:quillen}

In this section, we want to discuss some consequences of our computations for a conjecture of Quillen as well as detection questions in group cohomology. In particular, we are going to prove Theorem~\ref{thm:main3}. 

We first recall  the conjecture stated in Quillen's paper, cf. \cite{quillen:spectrum}*{p. 591}. 

\begin{conjecture}[Quillen] 
\label{Quillen-conjecture}
Let $\ell$ be a prime number. Let $K$ be a number field with $\zeta_\ell\in K$, and $S$ a finite set of places containing the infinite places and the places over $\ell$. Then the natural inclusion $\mathcal{O}_{K,S}\hookrightarrow \mathbb{C}$ makes $\op{H}^\bullet(\op{GL}_n(\mathcal{O}_{K,S}),\mathbb{F}_\ell)$ a free module over the cohomology ring $\op{H}^\bullet_{\op{cts}}(\op{GL}_n(\mathbb{C}),\mathbb{F}_\ell)\cong\mathbb{F}_\ell[c_1,\dots,c_n]$. 
\end{conjecture}

The range of validity of the conjecture has not yet been decided. Positive cases in which the conjecture has been established are $n=\ell=2$ by Mitchell \cite{mitchell}, $n=3$, $\ell=2$ by Henn \cite{henn}, and  $n=2$, $\ell=3$ by Anton \cite{anton}. 

A related question is the following detection of cohomology classes on diagonal matrices: 

\begin{definition}[Detection]
We say that detection of $\ell$-cohomology classes is satisfied for $\op{GL}_n(\mathcal{O}_{K,S})$ if the restriction morphism 
$
\op{H}^\bullet(\op{GL}_n(\mathcal{O}_{K,S}),\mathbb{F}_\ell)\to
\op{H}^\bullet(\op{T}_n(\mathcal{O}_{K,S}),\mathbb{F}_\ell)
$
is injective, where $\op{T}_n$ is the group of diagonal matrices in $\op{GL}_n$.
\end{definition}

Actually, all cases where the Quillen conjecture is known to be false can  be traced to \cite{henn:lannes:schwartz}*{remark on p. 51}, which shows that Quillen's conjecture implies detection for $\op{GL}_n(\mathbb{Z}[1/2])$. Non-injectivity of the restriction map, i.e., failure of detection, has been shown by Dwyer \cite{dwyer}  for $n\geq 32$ and $\ell=2$. Dwyer's bound was subsequently improved by Henn and Lannes to $n\geq 14$. At the prime $\ell=3$, Anton proved non-injectivity for $n\geq 27$, cf. \cite{anton}. 

Using the Farrell--Tate cohomology computations for $\op{SL}_2(\mathcal{O}_{K,S})$ in Theorem~\ref{thm:main1}, we can now discuss these questions - or  weaker versions - in the case $n=2$. Saying something about the Quillen conjecture means studying the module structure of $\op{H}^\bullet(\op{SL}_2(\mathcal{O}_{K,S}),\mathbb{F}_\ell)$ over the continuous cohomology ring $\op{H}^\bullet_{\op{cts}}(\op{SL}_2(\mathbb{C}),\mathbb{F}_\ell)\cong\mathbb{F}_\ell[c_2]$. From our computations of Farrell--Tate cohomology, we can infer statements on the module structure of $\widehat{\op{H}}^\bullet(\op{SL}_2(\mathcal{O}_{K,S}),\mathbb{F}_\ell)$, which allows for group cohomology statements above the virtual cohomological dimension. Recall from Theorem~\ref{thm:main1}, that the groups relevant for the computation of Farrell--Tate cohomology of $\op{SL}_2(\mathcal{O}_{K,S})$ are abelian groups $G=\mathbb{Z}/\ell\times\mathbb{Z}^n$ or dihedral extensions of such. The following result describes their Farrell--Tate cohomology as module over the relevant polynomial subrings. 

\begin{proposition}
\label{prop:free}
Let $G=\mathbb{Z}/\ell\times\mathbb{Z}^n$. 
\begin{enumerate}
\item Denoting by $b_1,x_1,\dots,x_n$ exterior classes of degree $1$ , and by $a_2$ a polynomial class of degree $2$, the cohomology ring 
$$
\op{H}^\bullet(G,\mathbb{F}_\ell)\cong
\mathbb{F}_\ell[a_2](b_1,x_1,\dots,x_n)
$$
is a free module of rank $2^{n+2}$ over the subring $\mathbb{F}_\ell[a_2^2]$.  
\item Let $\mathbb{Z}/2$ act via multiplication by $-1$ on all the generators. The invariant subring $\op{H}^\bullet(G,\mathbb{F}_\ell)^{\mathbb{Z}/2}$ is a free module of rank $2^{n+1}$ over the subring $\mathbb{F}_\ell[a_2^2]$.  
\end{enumerate}
\end{proposition}

\begin{proof}
(1) is clear from the explicit formula given in Proposition~\ref{6.1}; and $2^{n+2}$ is the rank obtained from the basis consisting of all the wedge products of the set 
$\{a_2, b_1,x_1,\dots,x_n\}$.
\\
(2) follows from this, the invariant ring is additively generated by $a_2^{\otimes(2i+1)}$ tensor the odd degree part of $\bigwedge(b_1,x_1,\dots,x_n)$ and $a_2^{\otimes(2i)}$ tensor the even degree part of $\bigwedge(b_1,x_1,\dots,x_n)$. 
\end{proof}

These statements now allow to formulate the following result, which we would like to see as a version of  the Quillen conjecture for $\op{SL}_2$ above the virtual cohomological dimension. Recall that $\op{H}^\bullet_{\op{cts}}(\op{SL}_2(\mathbb{C}), \mathbb{F}_\ell) \cong \mathbb{F}_\ell[c_2]$ is generated by the second Chern class $c_2$ (which is a class in degree $4$). This is the subring over which we have to express $\op{H}^\bullet(\op{SL}_2(\mathcal{O}_{K,S}),\mathbb{F}_\ell)$ as a free module for the Quillen conjecture. 

\begin{theorem}
Let $\ell$ be an odd prime number. Let $K$ be a number field, and let $S$ be a finite set of places containing the infinite places. Let $G = \op{SL}_2(\mathcal{O}_{K,S})$. In the splitting of Theorem~\ref{thm:main1},
$$
\widehat{\op{H}}^\bullet(\op{SL}_2(\mathcal{O}_{K,S}),\mathbb{F}_\ell)\cong
\bigoplus_{[\Gamma]\in\mathcal{K}_\ell}
\widehat{\op{H}}^\bullet(N_G(\Gamma),\mathbb{F}_\ell),
$$
denote by $n_1$ the number of components where $N_G(\Gamma)$ is abelian, and by $n_2$ the number of components where it is not. For each component group $N_G(\Gamma)$, denote by $a_{\Gamma}$ the second Chern class of the standard representation of $\Gamma$ (which is a polynomial class in degree $4$). 

Then the restriction map induced by the natural inclusion $\mathcal{O}_{K,S}\hookrightarrow \mathbb{C}$ is given as follows:
$$
\mathbb{F}_\ell[c_2]\to \op{H}^\bullet(\op{SL}_2(\mathcal{O}_{K,S}),\mathbb{F}_\ell) \to 
\widehat{\op{H}}^\bullet(\op{SL}_2(\mathcal{O}_{K,S}),\mathbb{F}_\ell):
c_2\mapsto \sum_{[\Gamma]\in\mathcal{K}_\ell} a_{\Gamma}.
$$
Denoting by $r$ the rank of the relative unit group of  $R_{K,S,\ell}/\mathcal{O}_{K,S}$, the Farrell--Tate cohomology
$\widehat{\op{H}}^\bullet(\op{SL}_2(\mathcal{O}_{K,S}),\mathbb{F}_\ell)$
is a free module of rank $2^{r+1}(2n_1+n_2)$ over the Laurent polynomial subring generated by the image of $c_2$. 
\end{theorem}

\begin{proof}
All the subgroups $\Gamma$ become conjugate in $\op{SL}_2(\mathbb{C})$ and the same is true for the centralizers $C_G(\Gamma)$. The restriction map for all groups $N_G(\Gamma)$ then has to map $c_2$ to the second Chern class of the standard representation of the cyclic or dihedral group, which is the element~$a_{\Gamma}$. The image of the restriction map is then the diagonal subring generated by $\sum a_{\Gamma}$. By Proposition~\ref{prop:free}, the cohomology ring is free as a module over this subring with the specified rank.
\end{proof}

As $\Gamma$ runs through finite cyclic groups, this proves the decomposition in into a sum of squares claimed in Theorem~\ref{thm:main3}.

\begin{corollary}
Let $K$ be a number field, let $S$ be a  finite set of places containing the infinite ones, and let $\ell$ be  an odd prime. 
\begin{enumerate}
\item
(An analogue of) The Quillen conjecture is true for the Farrell--Tate cohomology of  $\op{SL}_2(\mathcal{O}_{K,S})$. More precisely, the natural morphism   
$$
\mathbb{F}_\ell[c_2]\cong
\op{H}^\bullet_{\op{cts}}(\op{SL}_2(\mathbb{C}),\mathbb{F}_\ell)\to
\op{H}^\bullet(\op{SL}_2(\mathcal{O}_{K,S}),\mathbb{F}_\ell)
$$
extends to a morphism 
$$
\phi:\mathbb{F}_\ell[c_2,c_2^{-1}]\to
\widehat{\op{H}}^\bullet(\op{SL}_2(\mathcal{O}_{K,S}),\mathbb{F}_\ell)
$$
which makes $\widehat{\op{H}}^\bullet(\op{SL}_2(\mathcal{O}_{K,S}),\mathbb{F}_\ell)$
a free $\mathbb{F}_\ell[c_2,c_2^{-1}]$-module. 
\item The Quillen conjecture holds  for group cohomology $\op{H}^\bullet(\op{SL}_2(\mathcal{O}_{K,S}),\mathbb{F}_\ell)$ above the virtual cohomological dimension.  
\end{enumerate}
\end{corollary}

\begin{remark}
As a result, we can reformulate the Quillen conjecture for group cohomology as a relation between Farrell--Tate cohomology of $\op{SL}_2(\mathcal{O}_{K,S})$ and the Steinberg homology $\op{H}_\bullet(\op{SL}_2(\mathcal{O}_{K,S}), \op{St}^{\op{SL}_2(\mathcal{O}_{K,S})}\otimes\mathbb{F}_\ell)$. This follows from the long exact sequence relating group cohomology, Farrell--Tate cohomology and Steinberg homology, cf. \cite{brown:book}: 
$$
\cdots\rightarrow \widehat{H}^{\bullet-1}(\Gamma)\rightarrow
H_{n-\bullet}(\Gamma,\operatorname{St}^{\Gamma})\rightarrow
H^\bullet(\Gamma)\rightarrow 
\widehat{H}^\bullet(\Gamma)\rightarrow\cdots
$$
Vanishing of Steinberg homology for $\op{SL}_2(\mathcal{O}_{K,S})$ guarantees the Quillen conjecture by the above result; however, it is possible that the Quillen conjecture is true even with non-vanishing Steinberg homology. The Quillen conjecture fails whenever the map 
$$
\op{H}^\bullet(\op{SL}_2(\mathcal{O}_{K,S}),\mathbb{F}_\ell)\to
\widehat{\op{H}}^\bullet(\op{SL}_2(\mathcal{O}_{K,S}),\mathbb{F}_\ell)
$$
is not injective. Then there exist elements in group cohomology which - after multiplication with some power of the second Chern class - become trivial. 
\end{remark}

\begin{remark}
A similar result can be formulated for $\op{PGL}_2(\mathcal{O}_{K,S})$ if $\ell\in K$, but we chose not to spell it out explicitly. A result for $\op{GL}_2(\mathcal{O}_{K,S})$ is not as easy to come by, for the following reason. The Farrell--Tate cohomology for $\op{SL}_2(\mathcal{O}_{K,S})$ and $\op{PGL}_2(\mathcal{O}_{K,S})$ is controlled by finite subgroups and their normalizers. However, the central $\ell$-subgroup of $\op{GL}_2(\mathcal{O}_{K,S})$ fixes the whole symmetric space, so that computations of Farrell--Tate cohomology for $\op{GL}_2$ are actually not significantly easier than the group cohomology computations. 
\end{remark}

Next, we want to discuss the detection of cohomology classes, as well as its relation to the Quillen conjecture. Using our explicit computations, we can easily find examples where the cohomology of $\op{SL}_2(\mathcal{O}_{K,S})$ cannot be detected on the diagonal matrices. The following general result, deducing non-detection from non-triviality of suitable class groups, is very much in the spirit of Dwyer's disproof of detection for $\op{GL}_{32}(\mathbb{Z}[1/2])$.

\begin{proposition}
\label{prop:detect}
Let $\ell$ be an odd prime number. Let $K$ be a number field with $\zeta_\ell\in K$, and let $S$ be a finite set of places containing the infinite places.  Assume that the class group $\op{Pic}(\mathcal{O}_{K,S})$ has more than $2$ elements. Then the restriction map 
$
\widehat{\op{H}}^\bullet(\op{SL}_2(\mathcal{O}_{K,S}),\mathbb{F}_\ell)\to 
\widehat{\op{H}}^\bullet(\op{T}_2(\mathcal{O}_{K,S}),\mathbb{F}_\ell)
$
from $\op{SL}_2(\mathcal{O}_{K,S})$ to the diagonal matrices $\op{T}_2(\mathcal{O}_{K,S})$ is not injective. 
\end{proposition}

\begin{proof}
Under the assumptions of the proposition, Theorem~\ref{thm:conjsl2}(1)  implies that the splitting of Theorem~\ref{thm:main1} becomes 
$$
\widehat{\op{H}}^\bullet(\op{SL}_2(\mathcal{O}_{K,S}),\mathbb{F}_\ell)\cong 
\bigoplus_{[\Gamma]\in
  \op{Pic}(\mathcal{O}_{K,S})/\iota}
\widehat{\op{H}}^\bullet(N_G(\Gamma),\mathbb{F}_\ell).
$$
The group  $N_G(\Gamma)$ is the normalizer in $G = \op{SL}_2(\mathcal{O}_{K,S})$ of the order $\ell$ subgroup representing $[\Gamma]$. Recall that the Farrell--Tate cohomology $\widehat{\op{H}}^\bullet(N_G(\Gamma))$ is obtained by making the cohomology rings from Proposition~\ref{prop:free} periodic for $a_2$; in particular, all  exterior products with an odd number of factors live in odd degree, products with even number of factors live in even degree. In both cases, half of the elements is invariant under multiplication with $-1$. For the trivial ideal class, we have a contribution of half the rank of $\widehat{\op{H}}^\bullet(\op{T}_2(\mathcal{O}_{K,S}),\mathbb{F}_\ell)$.  Under the assumption on the class group, we either have two further $\iota$-invariant ideal classes or a $\iota$-orbit. In either case, the resulting direct summands in $\bigoplus_{[\Gamma]\in \op{Pic}(\mathcal{O}_{K,S})/\iota} \widehat{\op{H}}^\bullet(N_G(\Gamma),\mathbb{F}_\ell)$ yield a further contribution equal to the rank of $\widehat{\op{H}}^\bullet(\op{T}_2(\mathcal{O}_{K,S}),\mathbb{F}_\ell)$. Therefore, the restriction map cannot be injective  because the rank of the source is bigger than the rank of the target. 
\end{proof}
This completes the proof of Theorem~\ref{thm:main3}.

\begin{remark}
There are other cases in which non-detection results can be established. The above proposition is one of the easier ones to formulate, the cases where $\zeta_\ell\not\in K$ need some more complicated conditions. 
\end{remark}

\begin{example}
\label{ex:detect}
Let $K=\mathbb{Q}(\zeta_{23})$ and $S=\{(23)\}\cup S_\infty$. The $S$-class group of $K$ has order $3$. The induced morphism  
$$
\widehat{\op{H}}^\bullet(\op{SL}_2(\mathcal{O}_{K,S}),\mathbb{F}_{23})\to
\widehat{\op{H}}^\bullet(\op{T}_2(\mathcal{O}_{K,S}),\mathbb{F}_{23})
$$
is not injective - the source has two copies of the cohomology of a dihedral extension of $\mathcal{O}_{K,S}^\times$, but the target has only one copy of the cohomology of $\mathcal{O}_{K,S}^\times$. 
\end{example}

\begin{example}
\label{ex:detect2}
There are infinitely many counterexamples to detection at the prime~$3$. 

Let $m$ be a positive square-free integer such that $m\equiv 1\mod 3$. In this case, the prime $3$ is inert in the extension $\mathbb{Q}(\sqrt{-m})/\mathbb{Q}$ and ramified in $\mathbb{Q}(\zeta_3)/\mathbb{Q}$. In particular, there is only one place of $\mathbb{Q}(\sqrt{-m},\zeta_3)$ lying over the place $v_3$ of $\mathbb{Q}$, and this place is ramified in the extension $\mathbb{Q}(\sqrt{-m},\zeta_3)/\mathbb{Q}(\sqrt{-m})$. The extension $\mathbb{Q}(\sqrt{-m},\zeta_3)/\mathbb{Q}(\sqrt{-m})$ being ramified, the induced map on class groups must be injective, by class field theory. Moreover, as noted above, there is a unique prime ideal $\mathfrak{p}$ in $\mathcal{O}_{\mathbb{Q}(\sqrt{-m},\zeta_3)}$ lying above $(3)$, and $\mathfrak{p}^2=(3)$. The class group of $\mathcal{O}_{\mathbb{Q}(\sqrt{-m},\zeta_3)}[1/3]$ is obtained by killing the $2$-torsion class of $\mathfrak{p}$, hence its size is either equal or half the size of the class group of $\mathcal{O}_{\mathbb{Q}(\sqrt{-m},\zeta_3)}$ . 

We conclude that half the class number of $\mathbb{Q}(\sqrt{-m})$ is a lower  bound for the size of the class group of $\mathcal{O}_{\mathbb{Q}(\sqrt{-m},\zeta_3)}[1/3]$. By the theorem of Heilbronn, there are infinitely many $m$ such that $\mathbb{Q}(\sqrt{-m})$ has class number $>4$, hence there are infinitely many $S$-integer rings of the form $\mathcal{O}_{\mathbb{Q}(\sqrt{-m},\zeta_3)}[1/3]$ whose class group has more than two elements. Each such ring $R$ gives an example where the restriction map 
$$
\op{H}^\bullet(\op{GL}_2(R),\mathbb{F}_3)\to
\op{H}^\bullet(\op{T}_2(R),\mathbb{F}_3)
$$
fails to be injective, but for which the Quillen conjecture is true above the virtual cohomological dimension. 
\end{example}

The question for unstable analogues of the Quillen--Lichtenbaum conjecture was implicit in \cite{dwyer:friedlander}, and was raised explicitly in \cite{anton:roberts}. The results of \cite{dwyer:friedlander} show that the unstable Quillen--Lichtenbaum conjecture in the situation of linear groups over $S$-integers implies detection. The failure of detection as in Proposition~\ref{prop:detect} and Example~\ref{ex:detect} also implies the failure of the unstable Quillen--Lichtenbaum conjecture.  

The above results and their consequences for the Quillen conjecture are further discussed in \cite{RW14}. In short, the Quillen conjecture for Farrell--Tate cohomology is more related to the subgroup structure of $\op{SL}_2(\mathcal{O}_{K,S})$ and actually happens to be almost tautologically true in such small rank. On the other hand, the Quillen conjecture for group cohomology is more related to ``something like cusp forms'', lying in the difference between Farrell--Tate cohomology and group cohomology. Finally, detection questions, while a powerful method to provide counterexamples to Quillen's conjecture, are more related to the conjugacy classification of finite subgroups. Hence, for the rank one groups $\op{SL}_2(\mathcal{O}_{K,S})$, Quillen's conjecture and detection questions are only superficially related.

\section{Application II: on the existence of transfers}
\label{sec:transfers}

Next, we are interested in the existence of transfer maps in Farrell--Tate cohomology as well as group cohomology. Transfers in the cohomology of linear groups have been suggested as one way of establishing the Friedlander--Milnor conjecture in \cite{knudson:book}*{section 5.3}. In this section, we show examples that demonstrate the impossibility of defining transfers on group (co-)homology with reasonable properties. In particular, we are going to prove Theorem~\ref{thm:main4}. The general setup in this section will be the following: 
\begin{quote}
Let $L/K$ be a  degree $n$ extension of global fields, let $S$ be a set of places of $K$ and denote by $\tilde{S}$ the set of places of $L$ lying over $S$. Let $\ell$ be a prime different from the characteristic of $K$.
\end{quote}

\subsection{Definition of transfer maps}

We first recall various relevant notions of transfers. For a finite covering $p:E\to B$ of CW-complexes, there is a transfer map $\op{tr}_p:H_\bullet(B)\to H_\bullet(E)$ such that the respective composition is multiplication with the degree:  $\op{tr}_p\circ\, p_\bullet=\deg p$.  There are similar transfer maps for cohomology. More generally, the Becker--Gottlieb transfer provides a wrong-way stable map $\Sigma^\infty B\to \Sigma^\infty E$ which induces the transfer on cohomology theories.
This can be applied to group (co-)homology to recover the classical definition of transfer: for $H\subset G$ a finite-index subgroup of a group $G$ and a $G$-module $M$,  there are transfer maps 
$$
\op{tr}^G_H:\op{H}_\bullet(G,M)\to \op{H}_\bullet(H,M)\qquad
\textrm{ and }\qquad \op{Cor}^G_H:\op{H}^\bullet(H,M)\to\op{H}^\bullet(G,M),
$$
such that the following respective equalities hold where $i:H\to G$ denotes the inclusion of the subgroup:
$$
i_\bullet\circ\op{tr}^G_H=[G:H] \qquad\textrm{ and }\qquad
\op{Cor}^G_H\circ\, i^\bullet=[G:H]
$$
These can be obtained from the (usual or Becker--Gottlieb) transfer applied to the finite covering $BH\to BG$. 

Similar transfer maps can be considered in algebraic K-theory: for a finite flat map $f:X\to Y$ of schemes, there is a K-theory transfer $\op{tr}_f:\op{K}_\bullet(X)\to\op{K}_\bullet(Y)$ such that the composition is the multiplication with the degree: $\op{tr}_f\circ f^\ast=\deg f$. There are a number of generalizations of this concept. We use one of the simpler ones, cf. \cite{knudson:book}*{section 5.3}: an abelian group-valued functor $\mathcal{F}:\op{Sch}^{\op{op}}\to \mathcal{A}b$ is said to admit transfers if for any finite flat morphism $f:X\to Y$, there is a homomorphism $\op{tr}_f:\mathcal{F}(X)\to\mathcal{F}(Y)$ such that we have $\op{tr}_f\circ\mathcal{F}(f)=\deg f$.

We can apply this definition to group homology: for fixed $i\in\mathbb{N}$, we have a functor on the category of smooth affine $k$-schemes 
$$
\op{H}_i(\op{SL}_2(-),\mathbb{Z}/\ell):\op{Sm}^{\op{op}}_{\op{aff}}/k\to
\mathcal{A}b: 
\op{Spec} k[X]\mapsto \op{H}_i(\op{SL}_2(k[X]),\mathbb{Z}/\ell).
$$
One can ask a similar question for cohomology, i.e., if for any finite flat map $f:\op{Spec}k[X]\to\op{Spec}k[Y]$ of affine schemes there is a transfer homomorphism $\op{tr}^f:\op{H}^\bullet(\op{SL}_2(k[Y]),\mathbb{Z}/\ell)\to \op{H}^\bullet(\op{SL}_2(k[X]),\mathbb{Z}/\ell)$ such that the composition is multiplication with the degree: $\op{H}^\bullet(f)\circ\op{tr}^f=\deg f$. 
It is an implicit question in the discussion of \cite{knudson:book}*{section 5.3} if group homology functors as above can be equipped with transfers. To quote Knudson, cf. p.134 of loc.cit.: "Unfortunately, there appears to be no  way to equip these functors with transfer maps." In the following section, we want to make this precise: we will compute the restriction maps in Farrell-Tate cohomology, and from these computations it will be obvious that it is not possible to define transfer maps satisfying the degree condition. Similar counterexamples for function fields can be deduced from the computations of \cite{sl2parabolic}.

\subsection{Description of restriction maps}

To establish examples where transfers cannot exist, we need to determine the restriction maps for extensions of $S$-integer rings:
$$
\widehat{\op{H}}^\bullet(\op{SL}_2(\mathcal{O}_{L,\tilde{S}}),
\mathbb{F}_\ell)\rightarrow  
\widehat{\op{H}}^\bullet(\op{SL}_2(\mathcal{O}_{K,S}),\mathbb{F}_\ell).
$$

From Theorem~\ref{thm:main1},  we have a splitting of Farrell--Tate cohomology
$$
\widehat{\op{H}}^\bullet(\op{SL}_2(\mathcal{O}_{K,S}),\mathbb{F}_\ell)\cong
\bigoplus_{[\Gamma]\in\mathcal{K}}
\widehat{\op{H}}^\bullet(N_G(\Gamma),\mathbb{F}_\ell), 
$$
where the set $\mathcal{K}$ is given by the quotient of an
oriented relative class group modulo the Galois action, and the group
$N_G(\Gamma)$ is the  normalizer of the order $\ell$ subgroup
$\Gamma < G = \op{SL}_2(\mathcal{O}_{K,S})$.  
There are two types of components: if $\Gamma$ is not contained in a
dihedral group, then $N_G(\Gamma)$ is abelian, determined by a
relative unit group. Otherwise,  $N_G(\Gamma)$ is non-abelian
and it is a semidirect product of a relative unit group with the finite group
$\mathbb{Z}/2\mathbb{Z}$ acting via inversion.

Now, given an extension of  rings of $S$-integers $\mathcal{O}_{K,S}\to\mathcal{O}_{L,\tilde{S}}$, the following changes in the classification of finite subgroups can appear: 
\begin{enumerate}[({C }1)]
\item Several non-conjugate cyclic subgroups of order $\ell$ in  $\op{SL}_2(\mathcal{O}_{K,S})$ can become conjugate in $\op{SL}_2(\mathcal{O}_{L,\tilde{S}})$. 
\item A cyclic group which is not contained in a dihedral group in $\op{SL}_2(\mathcal{O}_{K,S})$ acquires a dihedral overgroup in $\op{SL}_2(\mathcal{O}_{L,\tilde{S}})$. 
\item There appear several new cyclic subgroups of order $\ell$ in $\op{SL}_2(\mathcal{O}_{L,\tilde{S}})$ which are not conjugate to subgroups coming from $\op{SL}_2(\mathcal{O}_{K,S})$. 
\end{enumerate}

The following proposition describes the restriction maps in Farrell--Tate cohomology in each of the above three cases; it is a straightforward consequence of the above statements, and Theorem~\ref{thm:main1}.

\begin{proposition} \label{8.1}
Fix an odd prime $\ell$. Let $K$ be a global field of characteristic different from $\ell$, let $S$ be a non-empty finite set of places containing the infinite ones. Let $L/K$ be a finite separable extension of $K$, and let $\tilde{S}$ be a finite set of places containing those places lying over $S$. 
\begin{enumerate}[({C }1)]
\item Assume that exactly the classes $[\Gamma_1],\dots,[\Gamma_m]\in\mathcal{K}(K)$ become identified to a single component $[\Gamma]\in\mathcal{K}(L)$. Then the restriction map 
$$
\widehat{\op{H}}^\bullet(
N_{\op{SL}_2(\mathcal{O}_{L,\tilde{S}})}(\Gamma),\mathbb{F}_\ell)\to  
\bigoplus_{[\Gamma_i]}
\widehat{\op{H}}^\bullet(N_{\op{SL}_2(\mathcal{O}_{K,S})}(\Gamma_i),\mathbb{F}_\ell)  
$$
is the sum of the natural maps  induced from the inclusions
\\
$N_{\op{SL}_2(\mathcal{O}_{K,S})}(\Gamma_i)\to N_{\op{SL}_2(\mathcal{O}_{L,\tilde{S}})}(\Gamma)$.  
\item Assume that a cyclic group representing the class $[\Gamma]\in\mathcal{K}(K)$ is not contained in a dihedral group over $K$, but is contained in a dihedral group over $L$. Then the restriction map 
$$
\widehat{\op{H}}^\bullet(N_{\op{SL}_2(\mathcal{O}_{L,\tilde{S}})}(\Gamma),\mathbb{F}_\ell)\to 
\widehat{\op{H}}^\bullet(N_{\op{SL}_2(\mathcal{O}_{K,S})}(\Gamma),\mathbb{F}_\ell) 
$$
is given by the natural inclusion of fixed points. 
\item The restriction map is trivial on the new components of $\mathcal{K}(L)$. 
\end{enumerate}
\end{proposition}
This yields claims (1) and (2) of Theorem \ref{thm:main4}.

\subsection{Non-existence of transfers}

\begin{theorem}
\label{thm:notr}
The functor ``Farrell--Tate cohomology of $\op{SL}_2$'' given by 
$$
A\mapsto \widehat{\op{H}}^\bullet(\op{SL}_2(A),\mathbb{F}_\ell)
$$ 
does not admit transfers. More precisely, there exists a finite flat morphism of commutative rings $\phi:A\rightarrow B$ such that for no morphism 
$$
\operatorname{tr}:
\widehat{\op{H}}^\bullet(\op{SL}_2(A),\mathbb{F}_\ell)\rightarrow 
\widehat{\op{H}}^\bullet(\op{SL}_2(B),\mathbb{F}_\ell)
$$
we have $\phi^\bullet\circ \operatorname{tr}=\deg\phi$.
\end{theorem}

\begin{proof}
Let $A\rightarrow B$ be any extension $\mathcal{O}_{K,S}\rightarrow \mathcal{O}_{L,\tilde{S}}$ of degree prime to $\ell$ such that situation (C 1) occurs with $m>1$. For example, we can take a prime $\ell$ and a number field $K$ with $\zeta_\ell\in K$ such that $\mathcal{O}_{K}$ has non-trivial class group and the  Hilbert class field $L$ for $\mathcal{O}_K$ has degree $[L:K]$ prime to $\ell$. An explicit example for this type of situation is given in Example~\ref{ex:q23} below.

Let $m>1$ be the number of classes of order $\ell$ subgroups which become conjugate to $\Gamma$. By Proposition~\ref{8.1}, the restriction map is  
$$
\widehat{\op{H}}^\bullet(
N_{\op{SL}_2(\mathcal{O}_{L,\tilde{S}})}(\Gamma),\mathbb{F}_\ell)\to 
\bigoplus_{i = 1}^{m}
\widehat{\op{H}}^\bullet(N_{\op{SL}_2(\mathcal{O}_{K,S})}(\Gamma_i),\mathbb{F}_\ell).
$$
The composition  $\phi^\bullet\circ \operatorname{tr}$ cannot be surjective, because the restriction map has its image contained in the diagonal subring. By the previous choice of $\ell$ and $[L:K]$, multiplication with the degree $[L:K]$ will have  full rank. Therefore, the two  maps cannot be equal, no matter how we choose $\operatorname{tr}$. 
\end{proof}

\begin{example}
\label{ex:q23}
To give a specific example, consider $K=\mathbb{Q}(\zeta_{23})$ with $S=S_\infty\cup S_{(23)}$. Its class number is $3$, and its group of units is $\mathbb{Z}^{11}\times\mathbb{Z}/46\mathbb{Z}$. Therefore, we have 
$$
\widehat{\op{H}}^\bullet(\op{SL}_2(\mathcal{O}_{\mathbb{Q}(\zeta_{23})}[1/23]),
\mathbb{F}_{23})\cong 
\widehat{\op{H}}^\bullet(\mathbb{Z}^{11}\times\mathbb{Z}/46\mathbb{Z},
\mathbb{F}_{23})\oplus
\widehat{\op{H}}^\bullet(\mathbb{Z}^{11}\times\mathbb{Z}/46\mathbb{Z},
\mathbb{F}_{23})^{\langle -1\rangle}
$$
i.e., the Farrell--Tate cohomology of $\op{SL}_2(\mathcal{O}_{\mathbb{Q}(\zeta_{23})}[1/23])$ is the direct sum of one  copy of the Farrell--Tate cohomology of the unit group and one copy of the $(-1)$-invariants of Farrell--Tate cohomology. Note that we have $2$ summands because this is the cardinality of the quotient of the class group $\mathbb{Z}/3\mathbb{Z}$ modulo the inversion involution. The Hilbert class field of $\mathbb{Q}(\zeta_{23})$  is a degree $3$ unramified extension $H/\mathbb{Q}(\zeta_{23})$ such that the restriction map $\op{Pic}(\mathcal{O}_{\mathbb{Q}(\zeta_{23})})\to \op{Pic}(H)$ is the trivial map. In particular, the restriction map on Farrell--Tate cohomology will factor through the diagonal map 
$$
\widehat{\op{H}}^\bullet(\op{T}_2(\mathcal{O}_{H}[1/23])\rtimes\mathbb{Z}/2, 
\mathbb{F}_{23})\to 
\widehat{\op{H}}^\bullet(\mathbb{Z}^{11}\times\mathbb{Z}/46\mathbb{Z},
\mathbb{F}_{23})\oplus
\widehat{\op{H}}^\bullet(\mathbb{Z}^{11}\times\mathbb{Z}/46\mathbb{Z},
\mathbb{F}_{23})^{\langle -1\rangle}
$$
No matter how the transfer map 
$$
\widehat{\op{H}}^\bullet(\op{SL}_2(\mathcal{O}_{\mathbb{Q}(\zeta_{23})}[1/23]),
\mathbb{F}_{23})\to
\widehat{\op{H}}^\bullet(\op{SL}_2(\mathcal{O}_H[1/23]),\mathbb{F}_{23})
$$
is defined, the composition with the restriction map will not be multiplication with~$3$ on  $\widehat{\op{H}}^\bullet(\op{SL}_2(\mathcal{O}_{\mathbb{Q}(\zeta_{23})}[1/23])$: multiplication with $3$ is invertible in $\mathbb{F}_{23}$, but there are many classes which are not in the image of the composition. 
\end{example}

\begin{remark}
Similarly, the situation (C 2) obstructs transfers because the inclusion of fixed points is not surjective. The situation (C 3) does not obstruct transfers. 
\end{remark}

\begin{corollary}
\label{cor:ntr}
The functor ``group cohomology of $\op{SL}_2$'', which is given by \\
\mbox{$A\mapsto \op{H}^\bullet(\op{SL}_2(A),\mathbb{F}_\ell)$}, does not admit transfers.   
\end{corollary}

\begin{proof}
Farrell--Tate and group cohomology agree above the virtual cohomological dimension, which for $\op{SL}_2$ over $S$-integers is finite. In the degrees above the vcd, the argument of Theorem~\ref{thm:notr} also applies to group cohomology.
\end{proof}
This concludes the proof of Theorem~\ref{thm:main4}.

\section{Application III: cohomology of \texorpdfstring{$\op{SL}_2$}{SL2} over number fields} 
\label{sec:borel}

Using the formulas for Farrell--Tate cohomology from Theorem~\ref{thm:main1} and restriction maps from Section~\ref{sec:transfers}, we can now compute the colimit of the Farrell--Tate cohomology groups over all possible finite sets of places. This could be interpreted as ``Farrell--Tate cohomology of $\op{SL}_2$ over global  fields''. The quotes are necessary as Farrell--Tate  cohomology is only defined for groups of finite virtual cohomological dimension. The goal of this section is to prove Theorem~\ref{thm:main5}.

\subsection{Recollection on Mislin--Tate cohomology} 

In  \cite{mislin:tate}, Mislin has defined an extension of Farrell--Tate cohomology to arbitrary groups.  The basic idea behind the definitions in \cite{mislin:tate} is to use satellites of group cohomology to kill projectives in the derived category of $\mathbb{Z}[G]$-modules, and obtain a completed cohomology, cf. \cite{mislin:tate}*{Section 2}. For groups of finite virtual cohomological dimension, Mislin's version of Tate cohomology agrees with Farrell--Tate cohomology, cf. \cite{mislin:tate}*{Lemma 3.1}. Moreover, Mislin shows that for $K$ a number field and $G=\op{GL}_n(K)$, his version of Tate cohomology can be identified with group homology, cf. \cite{mislin:tate}*{Theorem 3.2}. The same argument also applies to $\op{SL}_2$. 

\subsection{Colimit computations}

The following result is an immediate consequence of our earlier computations, in particular Theorem~\ref{thm:main1} and Proposition~\ref{8.1}. 

\begin{theorem}
\label{thm:limit}
Let $K$ be a number field, and let $\ell$ be an odd prime. We have the following three cases: 
\begin{enumerate}
\item If $\zeta_\ell+\zeta_\ell^{-1}\not\in K$, then 
$$
\lim_{S\supseteq S_\infty,
  {\rm finite}}\widehat{\op{H}}^\bullet(\op{SL}_2(\mathcal{O}_{K,S}),
\mathbb{F}_\ell)=0. 
$$
\item If $\zeta_\ell+\zeta_\ell^{-1}\in K$ and $\zeta_\ell\not\in K$, then 
$$
\lim_{S\supseteq S_\infty,
  {\rm finite}}\widehat{\op{H}}^\bullet(\op{SL}_2(\mathcal{O}_{K,S}),
\mathbb{F}_\ell)\cong \prod_{\Gamma\in \mathcal{K}_\ell}
\op{H}^\bullet(N_{\op{SL}_2(K)}(\Gamma),\mathbb{F}_\ell)[(a_2^2)^{-1}].
$$
In the above, $\mathcal{K}_\ell=\op{coker}\left(\op{Nm}_1:K(\zeta_\ell)^\times\to K^\times\right)/\op{Gal}(K(\zeta_\ell)/K)$ denotes the set of conjugacy classes of finite cyclic subgroups where the Galois action is the one described in Proposition~\ref{prop:galois}. The possible orientations of $K(\zeta_\ell)/K$ are given by the relative Brauer group $\op{coker}\left(\op{Nm}_1:K(\zeta_\ell)^\times\to K^\times\right)\cong\op{Br}(K(\zeta_\ell)/K)\cong H^2(\mathbb{Z}/2\mathbb{Z},K(\zeta_\ell)^\times)$. For  the normalizer, we have $N_{\op{SL}_2(K)}(\Gamma)\cong \op{T}(K)$ (the maximal torus of $\op{SL}_2(K)$) or $N_{\op{SL}_2(K)}(\Gamma)\cong \op{N}(K)$ (the normalizer of the maximal torus in $\op{SL}_2(K)$) depending on whether the class of $\Gamma$ is $\mathbb{Z}/2$-invariant or not. The $[(a_2^2)^{-1}]$ indicates that the polynomial class $a_2\in \op{H}^4(N_{\op{SL}_2(K)}(\Gamma),\mathbb{F}_\ell)$ is $\cup$-inverted.
\item If $\zeta_\ell\in K$, then 
$$
\lim_{S\supseteq S_\infty,
  {\rm finite}}\widehat{\op{H}}^\bullet(\op{SL}_2(\mathcal{O}_{K,S}),
\mathbb{F}_\ell)\cong\op{H}^\bullet(\op{N}(K),\mathbb{F}_\ell)[(a_2^2)^{-1}].
$$
\end{enumerate}
\end{theorem}


\begin{corollary}
There are cases where Mislin's version of Farrell--Tate cohomology does not commute with filtered colimits. 
\end{corollary}

\begin{proof}
For $K=\mathbb{Q}$, we have $\zeta_3+\zeta_3^{-1}\in\mathbb{Q}$ and $((\zeta_6-\zeta_6^{-1})^2,-1)=(-3,-1)=-1$. In terms of finite subgroups of $\op{PSL}_2(\mathbb{Q})$, there exist finite subgroups of order $3$, but they are not contained in dihedral groups. By Theorem~\ref{thm:limit} above, we have 
$$
\lim_{S\supseteq S_\infty,
  \textrm{finite}}
\widehat{\op{H}}^\bullet(\op{SL}_2(\mathcal{O}_{\mathbb{Q},S}),\mathbb{F}_3)= 
\op{H}^\bullet(\mathbb{Q}^\times,\mathbb{F}_3)[a_2^{-1}].
$$
In particular, the limit of Farrell--Tate cohomology groups in odd
degrees contains  
$$
\op{H}^1(\mathbb{Q}^\times,\mathbb{F}_3)\cong
\op{Hom}_{\mathbb{F}_3}(\mathbb{Q}^\times/(\mathbb{Q}^\times)^3,\mathbb{F}_3), 
$$
which is an infinite-dimensional $\mathbb{F}_3$-vector space. However, by
\cite{mislin:tate}*{Theorem 3.2} there is an isomorphism 
$$
\widehat{\op{H}}^1(\op{SL}_2(\mathbb{Q}),\mathbb{F}_3)\cong
\op{H}^1(\op{SL}_2(\mathbb{Q}),\mathbb{F}_3), 
$$
and the latter is trivial. Therefore, Mislin's version of Farrell--Tate
cohomology does not commute with directed colimits. 
\end{proof}

For an extension of number fields $L/K$, we can precisely describe the induced morphism on the colimit of the Farrell--Tate homology - it is induced by the inclusion $K^\times\hookrightarrow L^\times$. Again taking a colimit, we arrive at the following ``Farrell--Tate cohomology of $\op{SL}_2(\overline{\mathbb{Q}})$'': 

\begin{corollary}
\label{cor:sl2qbar}
Let $K$ be a number field and $\ell$ be an odd prime. We fix an algebraic closure $\overline{\mathbb{Q}}$ of $\mathbb{Q}$. Then we have 
$$
\lim_{\overline{\mathbb{Q}}\supseteq L\supseteq \mathbb{Q},
    S\supseteq  S_\infty}
\widehat{\op{H}}^\bullet(\op{SL}_2(\mathcal{O}_{L,S}),\mathbb{F}_\ell)
    \cong
\op{H}^\bullet(\op{N}(\overline{\mathbb{Q}}),\mathbb{F}_\ell)[(a_2^2)^{-1}],
$$
where $S$ runs through the finite sets of places of $L$ containing the
infinite places, and $\op{N}(\overline{\mathbb{Q}})$ denotes the
$\overline{\mathbb{Q}}$-points of the normalizer of a maximal torus of
$\op{SL}_2$. 
\end{corollary}

Finally, we want to explain how these colimits over Farrell--Tate cohomology groups allow to provide a reformulation of the Friedlander--Milnor conjecture. For a discussion of the Friedlander--Milnor conjecture, we refer to \cite{knudson:book}*{Chapter 5}. Recall that Milnor's form of what is now called the Friedlander--Milnor conjecture predicts that for a complex Lie group $G$, the natural change-of-topology map $\op{H}^\bullet_{\op{cts}}(G,\mathbb{F}_\ell)\to \op{H}^\bullet(G^\delta,\mathbb{F}_\ell)$ is an isomorphism, where the source is continuous cohomology of the Lie group $G$ with the analytic topology, and $G^\delta$ is the group with the discrete topology. On the other hand, Friedlander's generalized isomorphism conjecture predicts that for each algebraically closed field $K$ of characteristic different from $\ell$ and each linear algebraic group $G$ over $K$, another natural  change-of-topology map $\op{H}^\bullet_{\text{\'et}}(BG_K,\mathbb{F}_\ell)\to \op{H}^\bullet(BG(K),\mathbb{F}_\ell)$ is an isomorphism. From the rigidity property of \'etale cohomology, it follows that the Friedlander--Milnor conjecture for $\op{SL}_2$ over $\overline{\mathbb{Q}}$ is equivalent to the natural restriction map  
$$
\op{H}^\bullet_{\op{cts}}(\op{SL}_2(\mathbb{C}),\mathbb{F}_\ell)\to
\op{H}^\bullet(\op{SL}_2(\overline{\mathbb{Q}}),\mathbb{F}_\ell)
$$
being an isomorphism. Note that the continuous cohomology
$$
\op{H}^\bullet_{\op{cts}}(\op{SL}_2(\mathbb{C}),\mathbb{F}_\ell)\cong
\mathbb{F}_\ell[c_2]
$$
is a polynomial ring generated by the second Chern class. Therefore, the  Friedlander--Milnor conjecture for $\op{SL}_2$ over $\overline{\mathbb{Q}}$ can be reformulated as the claim that   $$\op{H}^\bullet(\op{SL}_2(\overline{\mathbb{Q}}),\mathbb{F}_\ell)
\cong
\mathbb{F}_\ell[c_2].
$$
Note that the right-hand side can also be identified with the ``Farrell--Tate cohomology of $\op{SL}_2(\overline{\mathbb{Q}})$'', by Corollary~\ref{cor:sl2qbar} above. This allows to reformulate the Friedlander--Milnor conjecture for $\op{SL}_2$ over $\overline{\mathbb{Q}}$ as the requirement that group cohomology and ``Farrell--Tate cohomology'' of $\op{SL}_2(\overline{\mathbb{Q}})$ are isomorphic: 

\begin{corollary}
The Friedlander--Milnor conjecture for $\op{SL}_2$ over  $\overline{\mathbb{Q}}$ with $\mathbb{F}_\ell$-coefficients, $\ell$ an odd prime, is equivalent to either one of the two following statements:  
\begin{enumerate}
\item
The colimit of homology groups of Steinberg modules vanishes: 
$$
\lim_{\overline{\mathbb{Q}}\supseteq L\supseteq \mathbb{Q},
    S\supseteq  S_\infty}
\op{H}_\bullet(\op{SL}_2(\mathcal{O}_{L,S}), 
\operatorname{St}_2(\mathcal{O}_{L,S});\mathbb{F}_\ell) = 0.
$$
where as above $S$ runs through the finite sets of places of $L$ containing the infinite places.
\item Mislin's version of Farrell--Tate cohomology commutes with the   filtered co\-limit over the intermediate fields $\overline{\mathbb{Q}}/L/\mathbb{Q}$ and their finite sets $S$ of places.
\end{enumerate}
\end{corollary}

\begin{proof}
We denote by $\mu_{\ell^\infty}$ the group of $\ell$-power roots of unity in $\overline{\mathbb{Q}}$. The group $\overline{\mathbb{Q}}^\times/\mu_{\ell^\infty}$ is uniquely $\ell$-divisible. In particular, the inclusion $\mu_{\ell^\infty}\subseteq \overline{\mathbb{Q}}^\times$ induces an isomorphism 
$$
\op{H}^q(\op{N}(\overline{\mathbb{Q}}),\mathbb{F}_\ell)\cong
\op{H}^q(\mu_{\ell^\infty}\rtimes\mathbb{Z}/2,\mathbb{F}_\ell)\cong
\left\{\begin{array}{ll}
\mathbb{F}_\ell ,& q\equiv 0\mod 4,\\
0, & \textrm{otherwise},\\
\end{array}\right.
$$
which is what is predicted by the Friedlander--Milnor conjecture. 

Let $G = \op{SL}_2(\mathcal{O}_{L,S})$. To prove Assertion (1), we take the filtered colimit of the long exact sequence 
$$
\cdots\rightarrow \widehat{\op{H}}^{\bullet-1}(G)\rightarrow
\op{H}_{n-\bullet}(G,\operatorname{St}^{G})\rightarrow
\op{H}^\bullet(G)\rightarrow 
\widehat{\op{H}}^\bullet(G)\rightarrow\cdots
$$
Group cohomology commutes with the filtered colimit, and by the above, the Friedlander--Milnor conjecture is equivalent to the fact that the colimit of the morphisms $\op{H}^\bullet(\op{SL}_2(\mathcal{O}_{K,S})\rightarrow \widehat{\op{H}}^\bullet(G)$ is an isomorphism. Since filtered colimits are exact for $\mathbb{F}_\ell$-vector spaces, the claim follows. 

Assertion (2) follows similarly. By \cite{mislin:tate}*{Theorem 3.2}, Mislin--Tate cohomology of $\op{SL}_2(\overline{K})$ agrees with group  cohomology, so the Friedlander--Milnor conjecture is equivalent to the commutation of Farrell--Tate cohomology with the filtered colimit.  
\end{proof}

\begin{remark}
Note that our computations above imply that ``Farrell--Tate cohomology of $\op{SL}_2(\overline{\mathbb{Q}})$'' is in fact detected on the normalizer of the maximal torus. This relates our result above to another reformulation of the Friedlander--Milnor conjecture: \cite{knudson:book}*{corollary 5.2.10} shows that the Friedlander--Milnor conjecture is equivalent to group cohomology being detected on the normalizer, cf. also the related discussion in \cite{sl2parabolic}. 
\end{remark}

\end{document}